\documentclass{article}

\usepackage[
	a4paper,
	margin=25mm
]{geometry}
\usepackage{setspace}
\usepackage{sectsty} 
\usepackage{titlesec}

\usepackage{amsmath, amssymb, amsthm}
\usepackage{mathtools}

\usepackage[hyphens]{url}
\usepackage[
	setpagesize=false,
	colorlinks=true,
	linkcolor=blue,
	pdfencoding=auto,
	psdextra,
]{hyperref}
\usepackage[hyphenbreaks]{breakurl}
\usepackage{cleveref}
\usepackage[
    sorting=nyt,
    maxbibnames=99
]{biblatex}

\usepackage{etoolbox} 
\usepackage[shortlabels]{enumitem}
\usepackage{xparse} 

\usepackage{graphicx}
\usepackage[dvipsnames]{xcolor}
\usepackage{tikz}

\usepackage{todonotes}
\usepackage[
    deletedmarkup=sout,
    commentmarkup=uwave,
    authormarkuptext=name
]{changes}

\setlist[enumerate,1]{label=(\arabic*), ref=(\arabic*)}
\setlist[enumerate,3]{label=(\roman*), ref=(\roman*)}

\allsectionsfont{\boldmath} 

\titlelabel{\thetitle.\quad}

\theoremstyle{plain}
\newtheorem{theorem}{Theorem}[section]
\newtheorem{lemma}[theorem]{Lemma}

\newtheorem{proposition}[theorem]{Proposition}
\newtheorem{conjecture}[theorem]{Conjecture}

\newtheorem{claim}{Claim}[theorem]
\newtheorem*{claim*}{Claim}

\makeatletter
\newenvironment{claimproof}[1][Proof]{\par
	\pushQED{\qed}%
	
	\normalfont \topsep6\p@\@plus6\p@\relax
	\trivlist
	\item[\hskip\labelsep
	\textit{#1}\@addpunct{.}~]\ignorespaces
}{%
	\popQED\endtrivlist\@endpefalse
}
\makeatother

\newcounter{case}
\AtBeginEnvironment{proof}{\setcounter{case}{0}}
\newcommand{\case}[1]{\refstepcounter{case}\par\medskip%
    \noindent\textbf{Case~\thecase:} #1.}
\crefname{case}{case}{cases}

\theoremstyle{definition}
\newtheorem{definition}[theorem]{Definition}
\newtheorem{remark}[theorem]{Remark}

\newenvironment{proofsketch}[1][Proof sketch]{%
	\begin{proof}[#1]
}{%
	\end{proof}
}

\newcommand{\NN}{\mathbb{N}}
\newcommand{\F}{\mathcal{F}}

\newcommand{\calA}{\mathcal{A}}
\newcommand{\calF}{\mathcal{F}}

\newcommand{\GG}{\mathcal{G}}

\DeclareMathOperator{\ex}{ex} 
\DeclareMathOperator{\Ex}{Ex} 

\makeatletter
\NewDocumentCommand{\xsideset}{mmme{_^}}{%
  \mathop{%
    \settowidth{\dimen0}{$\m@th\displaystyle#3$}%
    \dimen0=.5\dimen0
    \settowidth{\dimen2}{$%
      \m@th\displaystyle#3%
      \IfValueT{#4}{_{#4}}%
      \IfValueT{#5}{^{#5}}%
    $}%
    \dimen2=.5\dimen2
    \advance\dimen2 -\dimen0
    \sbox6{\scriptspace\z@$\displaystyle{\vphantom{#3}}#1$}
    \sbox8{\scriptspace\z@$\displaystyle{\vphantom{#3}}#2$}
    \ifdim\wd6>\dimen2 \kern\dimexpr\wd6-\dimen2\relax\fi
    {%
     \mathop{\llap{\copy6}{\displaystyle#3}\rlap{\copy8}}\limits
     \IfValueT{#4}{_{#4}}%
     \IfValueT{#5}{^{#5}}%
    }%
    \ifdim\wd8>\dimen2 \kern\dimexpr\wd8-\dimen2\relax\fi
  }%
}
\makeatother
\newcommand{\mins}[1]{\xsideset{}{_{#1}}\min} 

\newcommand{\symmdiff}{\mathbin{\triangle}}

\newcommand{\defeq}{\coloneqq}
\newcommand{\eqdef}{\eqqcolon}

\let\originalleft\left
\let\originalright\right
\renewcommand{\left}{\mathopen{}\mathclose\bgroup\originalleft}
\renewcommand{\right}{\aftergroup\egroup\originalright}

\makeatletter
\newcases{lrcases*}
  {\quad}
  {$\m@th{##}$\hfil}
  {{##}\hfil}
  {\lbrace}
  {\rbrace}
\makeatother

\usetikzlibrary{matrix,arrows,decorations.pathmorphing}
\usetikzlibrary{positioning}
\usetikzlibrary{graphs} 
\usetikzlibrary{graphs.standard}

\definechangesauthor[name=S, color=Emerald!60!black]{JSeo}
\definechangesauthor[name=K, color=blue!50!black]{JKim}
\definechangesauthor[name=C, color=orange!70!black]{D}

\makeatletter
\renewcommand*{\@textcolor}[3]{%
  \protect\leavevmode
  \begingroup
    \color#1{#2}#3%
  \endgroup
}
\makeatother

\title{On a rainbow extremal problem for color-critical graphs}

\author{%
    Debsoumya Chakraborti%
        \thanks{Discrete Mathematics Group (DIMAG), Institute for Basic Science (IBS), South Korea.
        E-mail: {\tt debsoumya@ibs.re.kr}.
        D.C. was supported by the Institute for Basic Science (IBS-R029-C1).}
    \and
    Jaehoon Kim%
        \thanks{Department of Mathematical Sciences, KAIST, South Korea.
        E-mail: {\tt \{jaehoon.kim, hyunwo9216, jaehyeon.seo\}@kaist.ac.kr}. J.K. was supported by the POSCO Science Fellowship of POSCO TJ Park Foundation.}
    \and
    Hyunwoo Lee\footnotemark[2]
    \and
    Hong Liu%
        \thanks{Extremal Combinatorics and Probability Group (ECOPRO), Institute for Basic Science (IBS), Daejeon, South Korea, Email: {\texttt hongliu@ibs.re.kr}. H.L. was supported by the Institute for Basic Science (IBS-R029-C4) and the UK Research and Innovation Future Leaders Fellowship MR/S016325/1.}
    \and
    Jaehyeon Seo\footnotemark[2]%
}

\begin{document}
\maketitle

\begin{abstract}
    There has been extensive studies on the following question: given $k$ graphs $G_1,\dots, G_k$ over a common vertex set of size $n$, what conditions on $G_i$ ensures a `colorful' copy of $H$, i.e., a copy of $H$ containing at most one edge from each $G_i$? 
    A lower bound on $\sum_{i\in [k]} e(G_i)$ enforcing a colorful copy of a given graph $H$ was considered by Keevash, Saks, Sudakov, and Verstra\"{e}te. They defined $\ex_k(n,H)$ to be the maximum total number of edges of the graphs $G_1,\dots, G_k$ on a common vertex set of size $n$ having no colorful copy of $H$. They completely determined $\ex_k(n,K_r)$ for large $n$ by showing that, depending on the value of $k$, one of the two natural constructions is always the extremal construction. Moreover, they conjectured the same holds for every color-critical graphs and proved it for 3-color-critical graphs. 

    
    
    We prove their conjecture for 4-color-critical graphs and for almost all $r$-color-critical graphs when $r > 4$. Moreover, we show that for every non-color-critical non-bipartite graphs, none of the two natural constructions is extremal for certain values of $k$. This answers a question of Keevash, Saks, Sudakov, and Verstra\"{e}te.
\end{abstract}

\section{Introduction}
For a given collection $\mathcal{F} =\{F_1,\dots, F_k\}$ of sets, a set $X\subseteq \bigcup F_i$ such that $|X\cap F_i|\leq 1$ for each $i\in [k]$ is often called a `colorful’ set of $\mathcal{F}$. If $|X\cap F_i|$ is exactly one for all $i$, then it is usually called a `transversal' of $\mathcal{F}$.
Colorful objects were considered over various types of mathematical objects.
B\'ar\'any considered a transversal of a collection of convex subsets of Euclidean spaces and obtained a colorful version of Carath\'eodory’s theorem in \cite{barany1982generalization}. Furthermore, more colorful variants of Carath\'eodory's theorem \cite{kalai2009colorful} and Helly's theorem \cite{kalai2005topological} were obtained. 
Aharoni and Howard considered transversals of a set systems and obtained colorful version of Erd\H{o}s-Ko-Rado theorem in \cite{aharoni2017rainbow}.

Perhaps the most famous transversals are the ones of Latin squares considered by Euler.
In 1782, Euler \cite{euler1782recherches} considered a Latin square, which is an $n\times n$ array filled with numbers $1,\dots, n$, where every number appears exactly once in each of the rows and columns. 
We may consider a Latin square as a collection of rows, columns, and the sets of entries with the same numbers, then a transversal in a Latin square is a set of $n$ entries such that no two are in the same row or in the same column or contain the same number. 

A transversal in a Latin square is a special instance of a colorful subgraph in edge-colored graphs. By considering rows and columns as the bipartition of $K_{n,n}$, a Latin square naturally corresponds to a proper edge-coloring of $K_{n,n}$. We can consider this edge-colored graph as a collection $\mathcal{G}= \{G_1,\dots, G_n\}$ of graphs where each $G_i$ is the graph consisting of edges of color $i$. A transversal of $\mathcal{G}$ which is also a matching itself, is exactly the transversal of a Latin square that Euler considered. See \cite{wanless2011transversals} for a survey.

This motivates studies about colorful subgraphs of a collection of graphs. Indeed, various interesting results have been proved. Aharoni, Devos, de la Maza, Montejano, and \v{S}\'{a}mal \cite{aharoni2020rainbow} considered a Tur\'an type problem over graph collections, proving that there exists a colorful triangle of $\{G_1,G_2,G_3\}$ if $\min_{i\in [3]} e(G_i)$ is bigger than $\frac{26-2\sqrt{7}}{81} n^2$. Surprisingly, this irrational number is best possible. Besides triangles, other graphs like perfect matchings, Hamilton cycles, and $F$-factors were considered and transversal version of Dirac’s theorem and Hajnal-Szemer\'{e}di theorems were obtained \cite{cheng2021rainbow,joos2020rainbow,montgomery2021transversal}.

In the above line of works, the results were obtained in terms of the restriction on $\min_{i} e(G)$ or $\min_{i} \delta(G)$. In other words, all graphs in the collection has to satisfy the given condition. However, enforcing conditions to all graphs in the graph collection seems quite restrictive. What if we reduce these conditions to the average behaviors of the graphs within the collection? In other words, in order to guarantee a colorful copy of $H$, how large $\sum_{i\in [k]} e(G_i)$ has to be?
Indeed, such a problem was already considered by a pioneering work of Keevash, Saks, Sudakov, and Verstra\"{e}te~\cite{keevash2004multicolour} and also recently by Frankl~\cite{frankl22}.
Keevash, Saks, Sudakov, and Verstra\"{e}te used different notion by considering this collection of graphs as one multi-graph with an edge-coloring  where the set of edges with each color corresponds to a simple graph within the collection. We follow the notion below introduced by them.

In this paper, a graph always means a simple graph.
A \emph{simple $k$-coloring} of a multigraph $G$ is a decomposition of the edge multiset as a disjoint sum of $k$ simple graphs which are referred as \emph{colors}. A sub-(multi)graph $H$ of a multigraph $G$ is called \emph{multicolored} if its edges receive distinct colors in a given simple $k$-coloring of $G$. If \(G\) contains a multicolored copy of \(H\), we would also say in short that \(G\) contains a multicolored \(H\). The \emph{$k$-color Tur\'{a}n number}, denoted by $\ex_k(n,H)$, is the maximum number of edges in an $n$-vertex multigraph that has a simple $k$-coloring containing no multicolored copy of $H$. The simply \(k\)-colored multigraphs that achieve this maximum are called the \emph{$k$-color extremal multigraphs} of \(H\). We denote the set of the extremal multigraphs by \(\Ex_k(n,H)\), but if there is only one such multigraph up to graph isomorphism, then we abuse the notation and also refer to it by \(\Ex_k(n,H)\).

If $k\leq e(H)-1$, then it is clear that $\Ex_k(n,H)$ is the multigraph consisting of $k$ copies of complete graphs. For $k\geq e(H)$, there are two natural maximal\footnote{maximal with respect to the subgraph relationship.} simply $k$-colored multigraph $G$ having no multicolored copy of $H$.
First, one can consider the multigraph consisting of $e(H) - 1$ copies of the complete graph. Secondly, one can consider $k$ identical copies of a fixed extremal $H$-free graph. The first multigraph has $(e(H)-1)\binom{n}{2}$ edges, whereas the second construction has $k\cdot\ex(n,H)$ edges. 
If $k$ is close to $e(H)$, then the first construction has more edges than the second one.
However, as $k$ grows (we allow $k$ to depend on $n$), for certain value of $k$ onward, the second construction has more edges than the first one. In \cite[Theorem~1.1]{keevash2004multicolour}, it was shown for the multicolor Tur\'{a}n problem for $H$ that whenever $k \ge \binom{n}{2} - \ex(n,H) + e(H)$, the second construction always gives the unique extremal multigraph. Keevash, Saks, Sudakov, and Verstra\"{e}te proved that when $H$ is a complete graph, the extremal multigraph is always one of these two natural constructions.  
It is well-known that the unique extremal $K_r$-free graph is the Tur\'an graph $T_{r-1}(n)$, which is the $n$-vertex balanced complete $(r-1)$-partite graph \cite{turan1954theory}.

\begin{theorem} [{\cite[Theorem~1.2]{keevash2004multicolour}}]
Suppose that $r \ge 2$, $k \ge \binom{r}{2}$, and $n > 10^4r^{34}$. Let $G$ be an $n$-vertex $k$-color extremal  multigraph of $K_r$. Then either all colors of $G$ are identical Tur\'an graphs $T_{r-1}(n)$, or there are exactly $\binom{r}{2} - 1$ non-empty colors of $G$, all of which are complete graphs $K_n$. In particular,
    \[
        \ex_k(n,K_r) = \begin{cases*}
            k \cdot t_{r-1}(n) & for \(k \ge \frac{1}{2}(r^2 - 1)\), \\
            \bigl(\binom{r}{2}-1\bigr)\binom{n}{2} & for \(\binom{r}{2} \le k < \frac{1}{2}(r^2 - 1)\).
        \end{cases*}
    \]
\end{theorem}

It is natural to consider $\ex_k(n,H)$ for more general graphs $H$. However, as the exact structures of the extremal graphs of $H$ are not known for most of the graphs $H$, it makes sense to first focus on the graphs $H$ whose extremal graphs are well-understood. An important class of such graphs is the class of color-critical graphs. A graph or a multigraph $H$ is called \emph{$r$-color-critical} if it has chromatic number $r$, and it has an edge (called a \emph{critical edge}) whose removal reduces the chromatic number to $r-1$. 
For an $r$-color-critical graph $H$, Simonovits~\cite[Theorem~1]{simonovits1968method} proved that $T_{r-1}(n)$ is the unique extremal graph of $H$ when $n$ is sufficiently large.
Indeed, Keevash, Saks, Sudakov, and Verstra\"{e}te conjectured that the above theorem can be extended to color-critical graphs.

\begin{conjecture}[{\cite[Conjecture~1.3]{keevash2004multicolour}}]\label{conj: conj}
    Suppose $r\geq 3$ and $k\geq h$. 
    Let $H$ be an \(r\)-color-critical graph with $h$ edges. Then, there exists an \(n_0=n_0(H)>0\) such that for all $n\ge n_0$, 
    the $k$-color extremal multigraph of $H$ either consists of exactly $h-1$ nonempty colors where each of them is a copy of $K_n$ or 
    consists of $k$ colors where all of them are identical copies of $T_{r-1}(n)$. In particular,
    \[
        \ex_k(n,H) = \begin{cases*}
            \left(h-1\right)\binom{n}{2} & for \(h \le k < \frac{r-1}{r-2}(h - 1)\),\\
            k \cdot t_{r-1}(n) & for \(k \ge \frac{r-1}{r-2}(h - 1)\).
        \end{cases*}
    \]
\end{conjecture}
The $r=3$ case of the conjecture was proved in \cite[Theorem~1.4]{keevash2004multicolour}. We further support the conjecture by proving the $r=4$ case. When $r \ge 5$, we also prove this conjecture for `most' of the $r$-color-critical graphs.

\begin{theorem} \label{thm:ex-gph_4-cc-gph}
    \Cref{conj: conj} holds for $r=4$.
\end{theorem}

\begin{theorem} \label{thm:ex-gph_almost-all_r-cc-gph}
    For $r\geq 5$ and a given $\varepsilon >0$, there exists $s_0$ such that the following holds for all $s\geq s_0$.
    At least $(1-\varepsilon)$-fraction of all $s$-vertex $r$-color-critical graphs $H$ on the vertex set $[s]$ satisfies \Cref{conj: conj}.
\end{theorem}

\begin{remark}\label{rmk:ex-gph_stab_good-r-cc-gph_1}
    We also prove a stability result for the above two theorems. To see the precise meaning of `stability' in this context, refer to \Cref{lem:ex-gph_stab_cc-gph_with_very-good-cred-mgph}.
\end{remark}

To prove \Cref{thm:ex-gph_4-cc-gph,thm:ex-gph_almost-all_r-cc-gph}, we develop the ideas in \cite{keevash2004multicolour}. In order to overcome certain technical difficulties, we consider certain $r$-vertex $r$-color-critical `multi'graph $H_c$ instead of a simple graph \(H\). 
By understanding the multicolor Tur\'{a}n problem for such a multigraph $H_c$, we are able to deal with the reduced (multi)graphs obtained by applying regularity lemma to $G_1,\dots, G_k$ and deduce our desired results \Cref{thm:ex-gph_4-cc-gph,thm:ex-gph_almost-all_r-cc-gph}. See \Cref{sec:prelims} for a rough sketch of the proofs of \Cref{thm:ex-gph_4-cc-gph,thm:ex-gph_almost-all_r-cc-gph}. 

It was asked in \cite{keevash2004multicolour} to identify the class of the graphs $H$ that have only the two extremal constructions as above. 
\Cref{thm:ex-gph_4-cc-gph,thm:ex-gph_almost-all_r-cc-gph} together with the result in \cite{keevash2004multicolour} show that all $r$-color-critical graphs for $r\in \{3,4\}$ and almost all $r$-color critical graphs for $r\geq 5$ lie in the class. What about non-color-critical graphs? In fact, we are able to show that, this class does not contain any non-color-critical graph with chromatic number $r \geq 3$. In contrast with \cite[Theorem~1.1]{keevash2004multicolour}, the condition of \(n\) being sufficiently large compared to \(k\) in this theorem is necessary.

\begin{proposition}\label{thm:ex-gph_non-cc-gph}
    Let $H$ be a non-color-critical graph with $h$ edges and chromatic number at least \(3\). Then, for any $k \ge k^*=\frac{(r-1)(h-1)}{r-2}$ and sufficiently large $n$, we have  $$\ex_k(n,H) > \max\left(k \cdot \ex(n,H), (h-1)\binom{n}{2}\right).$$ 
\end{proposition}

\paragraph{Organization.}
The rest of the paper proceeds as follows. We start by proving \Cref{thm:ex-gph_non-cc-gph} in \Cref{sec:construction_ex-gph_non-cc-gph}. We then give a few preliminaries along with a proof sketch of \Cref{thm:ex-gph_4-cc-gph,thm:ex-gph_almost-all_r-cc-gph} in \Cref{sec:prelims}. \Cref{thm:ex-gph_4-cc-gph,thm:ex-gph_almost-all_r-cc-gph} are proved through 
the next sections in the following way. In \Cref{sec:4-vtx-4-cc-mgph_good,sec:r-vtx-mgph_in_Fr_good}, we prove the results for an $r$-vertex $r$-color-critical `multi'graph. In other words, we respectively prove that all \(4\)-vertex \(4\)-color-critical multigraphs and most \(r\)-vertex \(r\)-color-critical multigraphs have two natural constructions as the multicolor extremal multigraphs. These results will be applied to regularity partition of $G_1,\dots, G_k$ to obtain approximate versions of \Cref{thm:ex-gph_4-cc-gph,thm:ex-gph_almost-all_r-cc-gph}. Appropriate stability versions of the results from \Cref{sec:4-vtx-4-cc-mgph_good,sec:r-vtx-mgph_in_Fr_good} will be proved in \Cref{sec:ex-gph_stab_good-mgph} and \Cref{sec:mcol-reg-lem}. 
 Finally, armed with everything we prove the exact statements of \Cref{thm:ex-gph_4-cc-gph,thm:ex-gph_almost-all_r-cc-gph} in \Cref{sec:very-good-cred-mgph=>very-good} along with their stability versions. In \Cref{subsec:most_r-cc-gphs_in_Fr}, we show some properties of most $r$-color-critical graphs which will be assumed before to prove \Cref{thm:ex-gph_almost-all_r-cc-gph}. We end with a few concluding remarks.

\section{Construction for non-color-critical graphs}\label{sec:construction_ex-gph_non-cc-gph}

\begin{proof}[Proof of \Cref{thm:ex-gph_non-cc-gph}]
Let $H$ be a non-color-critical graph with chromatic number $r$ for some $r \ge 3$. By Erd\H{o}s--Stone--Simonovits theorem, we know that $\ex(n,H) = t_{r-1}(n) + o(n^2)$. By simple computations, it is clear that for $k \ge k^*$ and sufficiently large $n$, we have that $k \cdot \ex(n,H) > (h-1)\binom{n}{2}$. Thus, to prove \Cref{thm:ex-gph_non-cc-gph}, it is enough to show that for $k \ge k^*$, we have that $\ex_k(n,H) > k \cdot \ex(n,H)$ for all sufficiently large $n$. Fix $k \ge k^*$. 

Consider the simply $k$-colored multigraph $G$ where all colors except one are equal to a fixed $T_{r-1}(n)$ and the final color is equal to $K_n$. We claim that this multigraph does not contain a multicolored copy of $H$. Suppose not, then we can embed a multicolor copy of $H$ in $G$ and fix such an embedding. It is clear that in the embedding, $H$ can have at most one edge outside of the fixed Tur\'an graph $T_{r-1}(n)$ (because there is only one color which contains edges outside of $T_{r-1}(n)$). Thus, there is an edge $e$ in $H$ such that $H \setminus e$ can be embedded in $T_{r-1}(n)$. Hence, $H \setminus e$ is an $(r-1)$-partite graph, contradicting to the fact that $H$ is a non-color-critical graph.

Finally, to finish the proof of \Cref{thm:ex-gph_non-cc-gph}, observe the following for sufficiently large $n$:
\[
    \ex_k(n,H) \ge e(G) = (k-1)\cdot t_{r-1}(n) + \binom{n}{2} > k\cdot \ex(n,H). \qedhere
\]
\end{proof}

\section{Preliminaries and necessary tools}
\label{sec:prelims}

Let $d_{r-1}(n)\defeq\delta(T_{r-1}(n))$ be the minimum degree of the Tur\'an graph $T_{r-1}(n)$.
We frequently use the following estimates: \(\frac{r-2}{r-1}(n-1) \le d_{r-1}(n)\le \frac{r-2}{r-1}\cdot n\) and \(\frac{r-2}{r-1}\binom{n}{2} < t_{r-1}(n)\le \frac{r-2}{r-1}\cdot \frac{n^2}{2}.\) 
We sometimes use the fact that \(e(H)\ge \binom{r}{2}\) for an \(r\)-color-critical multigraph \(H\), and in particular \(e(H)\ge 6\) if \(r\ge 4\). Throughout this paper, for brevity, we systematically avoid the floor and ceiling signs when they do not affect the underlying analysis.


	

\subsection{Notations}
In the following, \(n\), \(m\), \(r\), \(h\), and \(k\) always denote positive integers.
For a (multi)graph \(G\), let \(V(G)\) and \(E(G)\) denote the vertex set and the edge multiset of \(G\), respectively. We denote \(P(G)\defeq\binom{V(G)}{2}\). The multiplicity of an edge \(e\) in a multigraph $G$ is written as \(w_G(e)\) and the subscript will be omitted if the graph is clear from the context. For \(v\in V(G)\) and \(T\subseteq V(G)\), we write \(d_T(v)\defeq \sum_{u\in T} w_G(uv)\).


When we say a result holds if \(0<a\ll b,c\ll d<1\), it means that there exist non-decreasing functions \(f\) and \(g\) such that the result holds whenever \(b,c<f(d)\) and \(a<g(b,c)\). We will not compute those functions explicitly.

\subsection{Goodness and \texorpdfstring{\(\calF_r\)}{F\textrinferior}}
For the sake of convenience, we give a name to the property of having two natural constructions as the multicolor extremal multigraphs.

\begin{definition}
    We say a (multi)graph \(H\) with \(h\) edges and \(\chi(H)=r\) is \emph{good} if there exist \(n_0(H)>0\) such that the following holds for all \(k\ge h\), \(n\ge n_0(H)\), and $k^*=\frac{r-1}{r-2}(h-1)$. 
    If \(h\le k<k^*\), then an $n$-vertex $k$-color extremal multigraph of $H$ consists of exactly \(h-1\) non-empty colors, all of which are complete graphs \(K_n\), and if \(k\ge k^*\), then all colors of an $n$-vertex $k$-color extremal multigraph are the identical copies of an $H$-free graph.
\end{definition}

Then \Cref{conj: conj} says that for \(r\ge 3\), all the \(r\)-color-critical graphs are good. Also, \Cref{thm:ex-gph_non-cc-gph} says good graphs must be color-critical.
We denote the former extremal multigraph by \((h-1)K_n\). 
If \(H\) is a good \(r\)-color-critical graph with \(h\) edges, then $T_{r-1}(n)$ is the unique extremal graph of $H$, so we can denote the latter \(k\)-color extremal multigraph  by \(kT_{r-1}(n)\). 
It is easy to confirm that the latter one has more edges if and only if \(k\geq k^*= \frac{r-1}{r-2}(h-1)\) when \(n\) is sufficiently large.
For any \(r\)-color-critical (multi)graph \(H\) with \(h\) edges, we set \(k^*(H)=k^*(r,h)\defeq\frac{r-1}{r-2}(h-1)\). 

As mentioned in the introduction, we first study the problem for certain color-critical multigraphs.
To utilize this result to simple graph case, we make a connection between a $r$-color-critical graph $H$ and $r$-color-critical multigraph $H_c$ as follows. For an \(r\)-color-critical (multi-)graph \(H\), we call a proper coloring $f$ with color classes $V_1,\dots, V_r$ a \emph{critical} coloring if there exists two colors $\ell,\ell'$ such that $e(V_\ell, V_{\ell'})=1$.
For an \(r\)-color-critical (multi-)graph \(H\) and its critical coloring \(f\),
consider an \(r\)-vertex multigraph \(H^{f}\) whose vertices are the color classes \(V_1,\dots,V_r\) of \(f\) and the multiplicity of an edge \(V_iV_j\) is \(e(V_i,V_j)\) in \(H\) for each \(i\neq j\). 
By the choice of $f$, it is clear that $H^{f}$ is $r$-color-critical. 
We call \(H^{f}\) a \emph{color-reduced multigraph} of \(H\). 
Note that this choice depends on the choice of the coloring $f$.
For an \(r\)-color-critical (multi)graph \(H\), we write \(H_c\) to denote an $r$-color-critical multigraph $H^f$ where $\max\{ w_{H^f}(ij): ij\in P(H^f) \}$ is minimized over all choices of critical coloring of \(f\). If there are several choices of $f$ attaining the minimum, make an arbitrary choice.
This choice will be convenient to define $\mathcal{F}_r$ below.

We will show that for the \(4\)-color-critical graphs and most of the \(r\)-color-critical graphs $H$, \(r\ge 5\), their color-reduced multigraphs $H_c$ are also good. This fact together with the help of the multicolor version of the Szemer\'{e}di regularity lemma, 
we will prove that the original graphs $H$ are good as well. For this, we first establish a stability version of the corresponding color-reduced graph and then using an appropriate embedding lemma for multicolored regularity lemma to embed the targeted graph into a simply $k$-colored multigraph $G$ having more edges than the conjectured number.

In particular, for \(r\ge 5\), we prove the goodness for those graphs in a specific class called \(\calF_r\) which contains most of the \(r\)-color-critical graphs. An \(r\)-color-critical (multi)graph \(H\) with \(h\) edges is in \(\calF_r\) if it has an  $r$-color-critical color-reduced multigraph $H_c$ whose edge multiplicities are at most
\[
    \frac{2+2/r^2}{(r-1)(r-2)}(h-1)=\frac{1}{\binom{r-1}{2}}(h-1) + O\biggl(\frac{h}{r^4}\biggr).
\]
Note that a color-reduced multigraph \(H_c\) of \(H\in\calF_r\) is itself in \(\calF_r\). Roughly speaking, for $H\in \mathcal{F}_r$, the color-reduced multigraph \(H_c\) has balanced edge multiplicities, except for a few pairs including the pair having a critical edge. This allows to embed \(H_c\) in a certain \(r\)-vertex submultigraph \(G_0\) of a simply \(k\)-colored multigraph \(G\) with some lower bounds on edge multiplicities of \(G_0\). We will find such a \(G_0\) when \(G\) is multicolored-\(H\)-free and \(e(G)\ge\ex_k(|G|,H)\) but \(G\) does not have a desired structure, giving a contradiction.
Moreover, it is not difficult to show that almost all $r$-color-critical graphs belong to $\mathcal{F}_r$. We provide a proof sketch of the following proposition in \Cref{subsec:most_r-cc-gphs_in_Fr}.

\begin{proposition}\label{prop:almost all in Fr}
    For an integer $r$ and a real $\varepsilon>0$, there exists $s_0$ such that for all $s\geq s_0$ at least $(1-\varepsilon)$-fraction of all $s$-vertex $r$-color-critical graphs \(H\) on the vertex set \([s]\) are in \(\mathcal{F}_r\).
\end{proposition}

\subsection{Minimum degree condition of the host graph}
In many places, it is convenient to assume a minimum degree condition of the host graph \(G\). 
The following proposition allows us to assume such a minimum degree condition. We supply the proof in \Cref{appdx:(prop:min-deg-condition)-proof}.

\begin{proposition}\label{prop:min-deg-condition}
    Let $r\geq 3$, $k\geq 1$, and \(H\) be an $r$-color-critical (multi)graph with $h$ edges. For each $n\in\mathbb{N}$, let
\[
    A(n)\defeq\begin{cases*}
        (h-1)K_n & if \(h\le k<k^*(H)\), \\
        kT_{r-1}(n) & if \(h\ge k^*(H)\).
    \end{cases*}
\]
Suppose there is an \(M_0>0\) such that for all \(n>M_0\), this
$A(n)$ is the unique $n$-vertex simply \(k\)-colored extremal multigraph with at least $e(A(n))$ edges and minimum degree $\delta(A(n))$. Then there exists an \(n_0=n_0(M_0,k)>0\) such that for all \(n>n_0\), the multigraph \(A(n)\) is the unique $n$-vertex $k$-extremal multigraph of $H$.
\end{proposition}


The minimum degree condition obtained from the previous proposition helps us to find a vertex set whose induced graph contains edges with high multiplicities. Such an induced subgraph is useful to build a multicolored copy of \(H\). The following proposition allows us to find a vertex having many edges towards a fixed vertex set. Repeating this proposition yields a desired vertex set containing many edges with high multiplicities.

\begin{proposition}\label{prop:min-deg=>vtx_large-d_T(v)}
    Suppose \(0<\frac{1}{n}\ll \delta\ll \frac{1}{d},\frac{1}{t},\frac{1}{k}<1\).
    Suppose \(G\) is a simply \(k\)-colored multigraph of order \(n\) with \(\delta(G)\ge(1-\delta)d(n-1)\), and \(T\subseteq V(G)\) is a nonempty vertex set of order \(t\). Then there is a vertex \(v\in V(G)-T\) such that \(d_T(v)\ge dt\).
\end{proposition}
\begin{proof}
    We have that \(e(T,V(G)-T) = \sum_{v\in T}d(v)-2e(T) \ge t(1-\delta)d(n-1)-kt(t-1)\). Thus there is a vertex \(v\in V(G)-T\) such that \(d_T(v)\ge  \frac{1}{n-t}( t(1-\delta)d(n-1)-kt(t-1) )  > dt-1\).
\end{proof}

\subsection{Nested colorings}
We say that a simple \(k\)-coloring is \emph{nested} if its colors form a chain under inclusion. It is easy to see that the analogue of \cite[Proposition~2.1]{keevash2004multicolour} holds with the same proof even if we rather consider a multigraph \(H\) instead of a simple graph.

\begin{proposition}\label{prop:reform_to_nested-G}
    Suppose \(G\) is a simply \(k\)-colored multigraph, and \(G\) does not contain a multicolored (multi)graph \(H\). Then there exists a simply \(k\)-nested-colored multigraph \(F\) on the same vertex set as \(G\) such that
    \begin{enumerate}
        \item \(F\) and \(G\) have the same edge set as multigraphs, and
        \item \(F\) contains no multicolored \(H\).
    \end{enumerate}
\end{proposition}

The proof constructs \(F\) by applying the transformation to \(G\) finitely many times, which chooses two non-nested colors \(G_i\) and \(G_j\), and replace them by \(G_i\cap G_j\) and \(G_i\cup G_j\).
From this proposition, it is easy to see that for our purpose to prove that two natural constructions for $k$-extremal multigraphs cover all possibilities, it is enough to consider the nested colorings.
Especially, for a given color set \([k]\), we can assume that any edge of multiplicity \(s\) has colors exactly \([s]\), and we will use such assumptions in \Cref{sec:mcol-reg-lem} and \Cref{sec:very-good-cred-mgph=>very-good}.





\subsection{The condition for embedding a multicolored multigraph}
For two multigraphs \(H\) and \(G\), we say that an injective map \(\phi:V(H)\rightarrow V(G)\) is an \emph{embedding} if \(w_G(\phi(x)\phi(y)) \geq w_H(xy)\) for all \(xy \in P(H)\).
The following proposition allows us to  find a multicolored copy of a given multigraph.
This is a consequence of Hall's theorem, so we omit the proof.
\begin{proposition}\label{prop:proper-edge-emb-orders}
    Let \(G\) be a simply \(k\)-colored multigraph and \(H\) be an \(h\)-edge multigraph. Suppose there is an embedding \(\phi\colon H\hookrightarrow G\) which also gives \(P(H)\hookrightarrow P(G)\). Then \(\phi(H)\) yields a multicolored copy of \(H\) if there is an enumeration \((e_1,\dots, e_{h})\) of \(P(H)\) such that for each \(1\le j\le h\),
    \begin{equation}\label{eq:proper-edge-emb-orders_condition}
        \sum_{i=1}^{j} w_H(e_i) \le w_G(\phi(e_j)).
    \end{equation}
\end{proposition}
Note that the `only if' direction also holds if the simple $k$-coloring of $G$ is nested.
When we want to argue that \(G\) contains a multicolored \(H\), we aim to find an embedding \(\phi\colon H\hookrightarrow G\) and also an enumeration of \(P(H)\) which satisfy \eqref{eq:proper-edge-emb-orders_condition}. Note that among the embedding \(\phi\), the enumeration of \(P(H)\), and the corresponding enumeration of \(P(\phi(H))\), two of them determine the other. We sometimes call the enumerations the \emph{edge embedding orders}, and call them \emph{proper} if they satisfy \eqref{eq:proper-edge-emb-orders_condition}.

\subsection{\texorpdfstring{\(H\)}{H}-friendly submultigraphs}
When \(k\geq k^*(H)\), we want to prove that our host multigraph \(G\) is \((r-1)\)-partite. For this, we seek to find some `skeleton', which is an \((r-1)\)-partite submultigraph \(K\) of \(G\) containing a copy of \(H-\{x\}\) for some vertex \(x\in V(H)\). By analyzing how a vertex \(v\in V(G)\) and \(K\) interact, we can either obtain a multicolored copy of \(H\) within \(K\cup \{v\}\) or determine which part of the \((r-1)\)-partition does \(v\) belong to. The following concept of \emph{\(H\)-friendly submultigraph} provides such a skeleton structure that we need.

\begin{definition}
    Consider an \(r\)-color-critical (multi)graph \(H\). We say an \(a(r-1)\)-vertex simply \(k\)-colored \((r-1)\)-partite multigraph \(K\) with a vertex partition \(W_1,\dots, W_{r-1}\) of equal sizes is \emph{\(H\)-friendly} if \(K'\) obtained as follows always contains a multicolored \(H\): add a new vertex \(v\) to \(K\) and add edges of multiplicity at most \(k\) incident to \(v\) so that \(d(v)\ge \max\{(r-2)ak,(r-1)a(h-1)\} = (r-2)a\cdot \max\{k,k^*\}\) and \(d_{W_i}(v)\ge 1\) for all \(i\in [r-1]\).
\end{definition}

For example,  for $|H|=m$, a complete \((r-1)\)-partite graph \(h K_{m,\dots, m}\) is an \(H\)-friendly multigraph. Indeed, consider a critical edge \(xy\) of \(H\), and assume \(d_H(x)\leq d_H(y)\). Then it is easy to check that there exists a multicolored copy of \(H\) within \(K'=K\cup \{v\}\) as above where \(v\) plays the role of \(x\).

One thing to note is that the definition of \(H\)-friendliness depends on the choice of \((r-1)\)-partition \((W_1,\dots, W_{r-1})\) of \(K\). However, all \(H\)-friendly graphs we deal with in this paper will be complete \((r-1)\)-partite graphs with certain edge-multiplicities, which has the unique \((r-1)\)-partition.
Hence, we will not specify this vertex partition and just mention that such a multigraph \(K\) is \(H\)-friendly.

The following lemma states that the existence of an $H$-friendly submultigraph of $G$ implies the desired global structure of $G$.

\begin{lemma}\label{lem:H-friendly-submgph=>(r-1)-partite}
    Suppose  \(0\ll\frac{1}{n}\ll\delta\ll \frac{1}{a}, \frac{1}{m}, \frac{1}{k}\le 1\) and \(r\ge4\).
    Let \(H\) be an \(m\)-vertex \(r\)-color-critical multigraph with \(h\) edges and $k\geq k^*(H)$. Let \(G\) be a simply \(k\)-colored multicolored-\(H\)-free multigraph of order \(n\) with \(\delta(G)\ge(1-\delta)kd_{r-1}(n)\).
    
    If \(G\) contains an \(a(r-1)\)-vertex \(H\)-friendly multigraph \(K\) as an induced subgraph, then \(G\) is (\(r-1\))-partite.
\end{lemma}
    
\begin{proof}

    Write \(V\defeq V(G)\) and let \(W_1,\dots,W_{r-1}\) be a vertex partition of \(K\) making it \(H\)-friendly, with \(|W_i|=a\) for each \(i\).
    
    Suppose \(v\in V-V(K)\) satisfies \(\sum_{i=1}^{r-1} d_{W_i}(v)\ge (r-2)ak\). Then, as \(K\) is \(H\)-friendly and \(G\) contains no multicolored copy of \(H\), we have \(\min\{ d_{W_i}(v):1\le i\le r-1\}=0\). 
    In particular, since every edge has multiplicity at most $k$, this implies that $ d_{K}(v)\leq  (r-2)ak$ for all $i\in [r-1]$ and $v\in V-V(K)$.
    Let \(A\defeq\{v\in V-V(K):d_K(v)<(r-2)ak\}\). As $d_K(v)=(r-2)ak$ for all $v\in V-A-K$, we have
    \[
        \delta(G)|K|-k\cdot 2\binom{r-1}{2}a^2
        \le e(K,V-V(K))
        \le ((r-2)ak-1)|A|+(r-2)ak(n-|K|-|A|),
    \]
    which gives \(|A|\le (r-2)ak(\delta(n-1)+1)<\delta^{1/2}n\) by \(|K|=a(r-1)\) and the minimum degree condition on \(G\). The set \(V-A\) is partitioned into \((V_1,\dots, V_{r-1})\) where
    \[
        V_i\defeq\{v\in V-V(K)-A : d_{W_i}(v)=0,\ d_{W_j}(v)=ak\ \forall j\neq i\}. 
    \]
    We can check that for any \(v\in V_i\) and \(w\in W_i\), the submultigraph \(G[(V(K)\setminus\{w\})\cup \{v\}]\) is \(H\)-friendly. Indeed, as \(W_i\) is an independent set of \(K\), the degree sequence of \(G[(V(K)\setminus\{w\})\cup \{v\}]\) dominates that of \(K\), so it is \(H\)-friendly.
    As \(K\) is an induced subgraph of \(G\), each \(W_i\) is an independent set of \(G\). Thus, if there is an edge \(ab\) in \(V_i \cup W_i\), then one of \(a,b\) is in \(V_i\) and \(G[V(K)\cup\{a,b\}]\) contains a multicolored \(H\), a contradiction.
    Thus each \(V_i\cup W_i\) is an independent set.
    
    Now we show that we can place the vertices in \(A\) into one of the independent sets one by one while keeping them independent to conclude that \(G\) is (\(r-1\))-partite.
	Suppose to the contrary that at some step we have independent sets \(\{U_i\supseteq V_i\cup W_i: 1\le i\le r-1\}\) and there is \(v\in V-\bigcup_{i\in [r-1]}U_i\) such that \(d_{U_i}(v)\ge 1\) for each \(i\) so that we cannot place \(v\) in any of the \(U_i\)'s. Fix such a vertex \(v\) and let \(U\defeq\bigcup_{i\in [r-1]}U_i\).
	Note that we have \(|V-U|\le |A|\le \delta^{1/2}n\).
	
	If \(\bigl||U_i|-|U_j|\bigr|\ge 4\delta^{1/4}n\) for some \(i\neq j\), then
	\begin{align*}
		e(G)
		&\le e(U)+\sum_{v\in V-U}d(v) \le \sum_{i\neq j}|U_i||U_j| + \sum_{v\in V-U}d(v) 
		\\&\le \bigl(kt_{r-1}(|U|)-k(2\delta^{1/4}n)^2\bigr) + kn|V-U|
		< (1-\delta) kt_{r-1}(n),
	\end{align*}
	a contradiction as $\delta(G)\geq (1-\delta)k d_{r-1}(n)$. Thus \(\bigl||U_i|-\frac{|U|}{r-1}\bigr|<4\delta^{1/4}n\) for all \(i\), giving
	\[
	    \biggl||U_i|-\frac{n}{r-1}\biggr|<4\delta^{1/4}n+|V-U|<\delta^{1/5}n.
	\]
	
    Without loss of generality, let \(d_{U_1}(v)=\min\{d_{U_i}(v):1\le i\le r-1\}\). For each \(2\leq i\leq r-1\), let
    \[
        M_i\defeq\{u\in U_i : w(uv)\ge k/2\}\subseteq U_i
    \]
    
    \begin{claim}\label{cl: Mi size}
        For every \(2\leq i\leq r-1\) we have \(|M_i|>n/(4k^3)\).
    \end{claim}
    \begin{claimproof}
        Suppose that \Cref{cl: Mi size} does not hold for some \(2\leq i\leq r-1\). Then, since $r\leq h\leq k$, we have
    	\begin{align*}
    		d_{U_1}(v)
    		\le d_{U_i}(v)
    		&\le k|M_i| + \frac{k-1}{2}|U_i-M_i|
    		= \frac{k-1}{2}|U_i| + \frac{k+1}{2}|M_i|
    		\\&< \frac{k-1}{2}\biggl(\frac{n}{r-1}+\delta^{1/5}n\biggr) + \frac{k+1}{2}\cdot\frac{n}{4k^3}
    		\\&< \frac{k-1/8}{2(r-1)}n,
    	\end{align*}
    	which yields a contradiction, as the assumed minimality of \(d_{U_1}(v)\) yields
    	\begin{align*}
    		d(v)
    		&\le k(|V-U_1-U_i|)+d_{U_1}(v)+d_{U_i}(v)
    		\\&< k\biggl(n-2\biggl(\frac{n}{r-1}-\delta^{1/5}n\biggr)\biggr)
    		    + 2\cdot \frac{k-1/8}{2(r-1)}n 
    		\\&= \frac{r-2}{r-1}kn - \biggl(\frac{1}{8(r-1)}-2\delta^{1/5}k\biggr)n
    		< (1-\delta)kd_{r-1}(n). \qedhere
    	\end{align*}
    \end{claimproof}
    
    As \(d_{U_1}(v)\ge 1\), we can choose \(u_1\in U_1\) with \(w(vu_1)\ge 1\). Let \(W'_1\subseteq U_1\) be a set of \(m\) vertices containing \(u_1\).
	Consider a sequence $W'_1,\dots, W'_s$ of sets of \(m\) vertices with $s\leq r-1$ such that $W'_i\subseteq M_i$ for each $2\leq i\leq s$ and that $w(w_iw_j)=k$ for all $1\le j<i\le s$, $w_i\in W'_i$, and \(w_j\in W'_j\). Take such a sequence with the maximum possible $s$. Indeed, such a maximum choice exists as a sequence of one set \(W'_1\) trivially satisfies this condition with \(s=1\).
	
	We claim that $s=r-1$. Suppose $s<r-1$.
	For each $i\in [s]$, take
	\[
	    L_i\defeq\{u \in M_{s+1}: \exists w\in W'_i\ \text{s.t.}\ w(uw)<k\}.
	\]
	Then we have 
	\begin{align*}
	    m\cdot (1-\delta)k\frac{r-2}{r-1}(n-1)
	    &\le m\cdot (1-\delta)kd_{r-1}(n)
	    \\&\le \sum_{u\in W'_i} d(u)
	    < km(|V-U_{i}|-|L_i|)+(km-1)|L_i|
		\\&< km\biggl(n-\biggl(\frac{n}{r-1}-\delta^{1/5}n\biggr)\biggr)-|L_i|,
	\end{align*}
	giving $|L_i|\leq \delta^{1/6}n$. Thus
	\[
	    \biggl| M_{s+1} \setminus \bigcup_{i\in [s]}L_i \biggr| \geq \frac{n}{4k^3} - r\delta^{1/6}n> m.
	\]
	Hence, we may choose a set $W'_{s+1}$ of \(m\) vertices from \(M_{s+1} \setminus \bigcup_{i\in [s]}L_i\) satisfying $w(uw)=k$ for all \(u\in W'_{s+1}\), \(w\in W'_i\), and $i\in [s]$, a contradiction to the maximality of $s$. Hence we have \(s=r-1\).
	
	Let \(xy\) be a critical edge of \(H\). Then \(d(x)+d(y)+\binom{r-2}{2}\le h+1\), so without loss of generality let \(d(y)\le h/2<k/2\). This allows an embedding of a multicolored \(H\) into \(G[\bigl(\bigcup_{i\in [r-1]}W'_i\bigr)\cup\{v\}]\) with \(x\mapsto u_1\) and \(y\mapsto v\), using \Cref{prop:proper-edge-emb-orders} with an embedding order of \(P(H)\) starting with the edges incident to \(y\), a contradiction.
\end{proof}

\subsection{Ignoring large \texorpdfstring{\(k\)}{k}}
The final proposition of this section is an analogue of \cite[Proposition~2.5]{keevash2004multicolour}, which helps us to reduce the multicolor Tur\'an problems for sufficiently large $k$ to a single value of $k$. The same proof works almost line by line.
\begin{proposition}\label{prop:k-large_enough-to-check_k=k*}
    Let \(H\) be a graph.
    Suppose that there is an \(n_0=n_0(k)>0\) such that for \(n>n_0\), there exists a fixed $H$-free simple graph \(F_n\) of order \(n\) such that every $n$-vertex $k$-extremal multigraph \(G\) of \(H\) consists of $k$ identical copies of $F_n$.  Then  for all \(l\ge k\) and \(n>n_0\), every $n$-vertex $\ell$-extremal multigraph \(G\) of \(H\) consists of $\ell$ identical copies of \(F_n\).
\end{proposition}
This allows us that in the case of \(k\ge k^*(H)\), we can only  consider the case of \(k=\lceil k^*(H)\rceil\), so that the lower bound of \(n\) only depends on \(H\) rather than both \(H\) and \(k\).

\section{4-vertex 4-color-critical multigraphs are good}
\label{sec:4-vtx-4-cc-mgph_good}

In this section, we prove the following theorem.

\begin{theorem}\label{thm:ex-gph_4-vtx-4-cc-mgph}
    Every \(4\)-vertex \(4\)-color-critical multigraph is good.
\end{theorem}

We start with a proof sketch of this result. Let \(H\) be a \(4\)-vertex \(4\)-color-critical multigraph with \(h\) edges and the fixed critical edge, and let \(G\) be a rainbow-\(H\)-free multigraph of order \(n\) such that
\[
    e(G)\ge \begin{cases*}
        (h-1)\binom{n}{2} & if \(k<k^*=k^*(H)=\frac{3}{2}(h-1)\), \\
        kt_3(n) & if \(k\ge k^*\).
    \end{cases*}
\]

First consider the case \(k<k^*\). By \Cref{prop:min-deg-condition}, we may assume that \(\delta(G)\ge \delta((h-1)K_n)=(h-1)(n-1)\). If the multiplicities of every edge is at most $h-1$, then there is nothing to prove. Otherwise, there is an edge \(v_1v_2\) of multiplicity at least \(h\). Next we utilize \Cref{prop:min-deg=>vtx_large-d_T(v)} to obtain a pair of vertices $v_3$, $v_4$ such that the following equations hold:
\[
    w(v_1v_3)+w(v_2v_3)\ge 2(h-1) \quad \text{and} \quad w(v_1v_4)+w(v_2v_4)+w(v_3v_4)\ge 3(h-1).
\]
It turns out that in $H$, if the multiplicity of the edge not adjacent to the critical edge is not too large, then we can embed $H$ in a greedy manner using \Cref{prop:proper-edge-emb-orders}. In the remaining case, we argue that there is a set $S$ consisting of at least half of the vertices $v$ in $G$ with `large' value of $w(v_1v) + w(v_2v)$. The `large' number will be chosen in such a way that $S$ will be forced to be an independent set (otherwise existence of an edge $e$ within $S$ will enable us to embed $H$ with its critical edge being embedded into $e$). However, this will violate the minimum degree assumption on $G$, giving us a contradiction.

For the case \(k \ge k^*\), by \Cref{prop:min-deg-condition}, we assume that $\delta(G)\ge \delta(k T_3(n))$. Similar arguments as above can be made to find an edge \(v_1v_2\) with multiplicity \(k\) and then find a vertex \(v_3\) such that if \(v\in V-\{v_1,v_2,v_3\}\) satisfies \(e(v,\{v_1,v_2,v_3\})\ge 2k\), then it is exactly \(2k\) and \(\{w(v_1v),w(v_2v),w(v_3v)\}=\{0,k,k\}\) as a multiset. This forms a \(3\)-vertex \(H\)-friendly subgraph of \(G\). Using this as a skeleton of the tripartite structure of \(G\), we can obtain the exact structure of $G$.

As we mentioned above, in the case when the multiplicity of the edge not adjacent to the critical edge of $H$ is not too large, we can embed $H$ with a rather involved case-study. The next lemma serves that purpose for us. For the convenience of stating it, we introduce one notation. Let \(p_1,\dots,p_\ell\) be real numbers. For each \(j\in[\ell]\), we denote the \(j\)-th smallest number among \(\{p_1,\dots,p_\ell\}\) by \(\mins{j}\{p_i:1\le i\le\ell\}\). In particular, if $p_1\leq \dots \leq p_\ell$, then $\mins{j}\{p_i\} = p_j$.

\begin{lemma}\label{lem:4-vtx-4-cc-mgph_small-case}
    Let \(H\) be a \(4\)-vertex \(4\)-color-critical multigraph with \(h\) edges, and \(k\ge h\). Let \(G_0\) be a simply \(k\)-colored \(4\)-vertex multigraph. Suppose the edge multiplicities of \(H\) and \(G_0\) are given as in \Cref{fig:(lem:4-vtx-4-cc-mgph_small-case)_labeling_G0_H}, which satisfy
	\begin{gather*}
		a\ge h;\qquad
		b_1\ge \frac{h-1}{2},\qquad b_2\ge h-1; \\
		c_1+c_2+c_3\ge \max\{3h-3,2k\},\qquad
			\min\{c_i\}\ge 1,\qquad \mins{2}\{c_i\}\ge \frac{3h-3}{4},\qquad \max\{c_i\}\ge h-1.
	\end{gather*}
	Except the case when \(h_5\ge b_1\), there is a multicolored copy of \(H\) in \(G_0\).
	
	\begin{figure}[!htbp]
		\centering
		\begin{tikzpicture}[
			x=3cm, y=3cm,
			every node/.style={draw, circle, fill=white},
			labelnode/.style={draw=none, rectangle},
			]
			\draw (0,0) node(v1){\(v_1\)}
				-- node[labelnode]{\(a\)} (0,1) node(v2){\(v_2\)}
				-- node[labelnode]{\(b_2\)} (1,1) node(v3){\(v_3\)}
				-- node[labelnode]{\(c_3\)} (1,0) node(v4){\(v_4\)}
				-- node[labelnode]{\(c_1\)} (v1);
			
			\draw (v1) -- node[labelnode,near end]{\(b_1\)} (v3)
				(v2) -- node[labelnode,near end]{\(c_2\)} (v4);
				
			\node[labelnode] at (0.5,-0.25) {$G_0$};
		\end{tikzpicture}
		\qquad
		\begin{tikzpicture}[
			x=3cm, y=3cm,
			every node/.style={draw, circle, fill=white},
			labelnode/.style={draw=none, rectangle},
			]
			\draw (0,0) node(x1){\(x_1\)}
			-- node[labelnode]{\(I=1\)} (0,1) node(x2){\(x_2\)}
			-- node[labelnode]{\(h_2\)} (1,1) node(x3){\(x_3\)}
			-- node[labelnode]{\(h_5\)} (1,0) node(x4){\(x_4\)}
			-- node[labelnode]{\(h_3\)} (x1);
			
			\draw (x1) -- node[labelnode,near end]{\(h_1\)} (x3)
			(x2) -- node[labelnode,near end]{\(h_4\)} (x4);
			
			\node[labelnode] at (0.5,-0.25) {$H$};
		\end{tikzpicture}
		\caption{The labelings of \(G_0\) and \(H\).}
		\label{fig:(lem:4-vtx-4-cc-mgph_small-case)_labeling_G0_H}
	\end{figure}
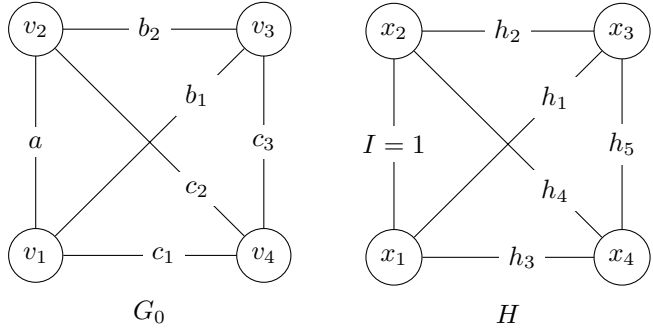
\end{lemma}
\begin{proof}
	For each edge, we use the label referring its multiplicity to refer the edge itself if there is no risk of confusion, and we denote \(I\defeq x_1x_2\).
	
	We consider several cases, and for each case we determine the proper edge embedding orders using \Cref{prop:proper-edge-emb-orders}. We will decide two enumerations \((e_1,\dots,e_6)\) and \((f_1,\dots,f_6)\) of \(P(H)\) and \(P(G_0)\), respectively, and embed each \(e_i\) into \(f_i\). Except the last one case, we will set 
    \[
        e_1=I
        \quad \text{and} \quad (f_i)_{i=1}^6=\bigl(\min\{c_i\},b_1,\mins{2}\{c_i\},\max\{c_i\},b_2,a\bigr).
    \]
	The embedding orders will be determined so that they are consistent with respect to the vertex-edge incidence relations and they satisfy the condition in \Cref{prop:proper-edge-emb-orders}. 
	In particular, since \((f_i)_{i=1}^6\) is fixed, if we determine the positions of two incident edges within $(e_1,\dots, e_6)$, this determines how three vertices (thus, indeed all four vertices) must be mapped into $G_0$ and determine the entire sequence $(e_1,\dots, e_6)$. This concludes the following useful fact.
    \begin{equation}\label{eq: uniquemap}
        \begin{minipage}{0.9\textwidth}
        For two incident edges $f_i, f_j$ of $G_0$ and two incident edges $e_i, e_j$ of $H$, there exists exactly one embedding $\phi: H\hookrightarrow G_0$ such that $\phi(e_i)=f_i$ and $\phi(e_j)=f_j$. 
        \end{minipage}
    \end{equation}
    Note that the above choice of \((f_1,\dots, f_6)\) implies 
    \begin{equation}\label{eq: granted_order}
        f_1 \geq 1= I,
        \quad
        f_4, f_5\geq h-1 \geq h-e_6 = e_1+\dots+ e_5,
        \quad
        \text{and}
        \quad
        f_6\geq h.
    \end{equation}
    Thus, to verify that the orders we decide is proper, we only need to check 
    \begin{equation}\label{eq: proper}
        f_2\geq 1+ e_2
        \quad \text{and} \quad
        f_3 \geq 1+e_2+e_3.
    \end{equation}
	
	\case{\(h_1,\dots,h_5<(h-1)/2\)}
	In this case, as every \(h_i\) is at least \(1\) and they sum up to \(h-1\), we see 
	\[
	    \min\{h_i\}\le \frac{h-1}{5}
	    \quad \text{and} \quad
	    \mins{2}\{h_i\}\le \frac{h-2}{4}.
	\] 
	We consider two subcases as follows.
		
	\begin{enumerate}[label=\textbf{Case~\thecase}-(\arabic*):, wide,nosep]
		\item \(f_1 = \min\{c_i\} \neq c_2\).
		    In this case, we let \(e_2\defeq\min\{h_1,\dots,h_4\}\).
		    As $e_1$ and $e_2$ are fixed and they are incident edges of $H$ and $f_1$ and $f_2$ are also incident in $G_0$,  \eqref{eq: uniquemap} yields a unique embedding.\footnote{For example, if \(f_1=c_1\) and \(e_2=h_1\), then we should embed \(x_1,x_2,x_3,x_4\) to \(v_1,v_4,v_3,v_2\) in this order. So, \((e_i)_i=(I,h_1,h_4,h_2,h_5,h_3)\) if \(\mins{2}\{c_i\}=c_2\) and \((e_i)_i=(I,h_1,h_2,h_4,h_5,h_3)\) otherwise.}
		    Moreover, this yields proper embedding orders because the orderings satisfy \eqref{eq: proper} as
		    \begin{align*}
		        f_2 & \ge \frac{h-1}{2}\ge 1+\mins{2}\{h_i\} \quad \text{and} \\
		        f_3 &
				\ge \left\lceil\frac{3h-3}{4}\right\rceil
				\ge 1 +\left\lfloor\frac{h-1}{2}\right\rfloor +\left\lfloor\frac{h-2}{4}\right\rfloor
				\ge 1+\max\{h_i\}+\mins{2}\{h_i\}
				\ge 1+e_2+e_3.
		    \end{align*}
		    Here, one can check the second inequality for $f_3$ by considering the value of \(h \bmod 4\).
			
		\item \(f_1= \min\{c_i\} =c_2\).
			In this case, we let \(e_2\defeq h_5\) and \(e_3\defeq\min\{h_1,\dots,h_4\}\). As $e_1$ and $e_3$ are fixed and they are incident edges of $H$ and $f_1$ and $f_3$ are also incident in $G_0$, \eqref{eq: uniquemap} yields a unique embedding.
			Moreover, this yields proper embedding orders because the orderings satisfy \eqref{eq: proper} as
			\begin{align*}
			    f_2 &\ge \left\lceil\frac{h-1}{2}\right\rceil \ge 1+h_5 \quad \text{and} \\
			    f_3
				&\ge \left\lceil\frac{3h-3}{4}\right\rceil
				\ge 1+\left\lfloor\frac{h-1}{2}\right\rfloor
				    +\left\lfloor\frac{h-1}{5}\right\rfloor
				\ge 1+\max\{h_i\}+\min\{h_i\}
				\ge 1+e_2+e_3.
			\end{align*}
			Here, by noting that \(\left\lfloor\frac{h-2}{4}\right\rfloor \ge \left\lfloor\frac{h-1}{5}\right\rfloor\), the second inequality for $f_3$ is implied by the last inequality used in the previous case.
	\end{enumerate}
	
	\case{\(\max\{h_1,\dots,h_4\}\ge (h-1)/2\)}
	In this case, without loss of generality suppose that 
	\[
	    h_1=\max\{h_1,\dots,h_4\}\ge (h-1)/2
	    \quad \text{and thus} \quad
	    h_2 + h_3 + h_4 + h_5\le (h-1)/2.
	\]
	Hence, we have
	\[
	    \mins{3}\{h_2,\dots,h_5\}\le \frac{h-5}{4} 
	    \quad \text{and} \quad
        \max\{h_2,\dots,h_5\}\le \frac{h-7}{2}.
    \]
    
    If \(f_1 = \min\{c_i\} \neq c_3\), then we set \(e_6\defeq h_1\) so that we will embed \(h_1\) to \(a\) which has the largest multiplicity lower bound. As $e_6=h_1$ is incident with $e_1=I$ and $f_1$ and $f_6$ are incident in $G_0$, \eqref{eq: uniquemap} yields a unique embedding.
    
    If \(f_1=\min\{c_i\}=c_3\), then we want to embed \(I\) to \(c_3\). As \(h_1\) is incident with \(I\), we cannot embed \(h_1\) to \(a\) which is not incident with \(c_3\). Hence we let \(e_5\defeq h_1\) so that we embed \(h_1\) to the edge of \(G_0\) whose multiplicity lower bound is the second largest. 
    As $e_5$ is incident with $e_1=I$ and $f_1$ and $f_5$ are incident in $G_0$, \eqref{eq: uniquemap} yields a unique embedding.

    In either case, the embedding orders are proper embedding orders because \eqref{eq: proper} holds as
	\begin{align*}
		f_2&\ge \frac{h-1}{2} \ge 1+\frac{h-7}{2} \ge 1+\max\{h_2,\dots,h_5\}\ge 1+e_2
		\quad\text{and} \\
		f_3&\ge \frac{3h-3}{4} \ge 1+\frac{h-7}{2}+\frac{h-5}{4} \ge 1+\max\{h_2,\dots,h_5\}+\mins{3}\{h_2,\dots,h_5\}\ge 1+e_2+e_3.
	\end{align*}
	
	\case{\(h_5\ge (h-1)/2\) and \(b_1\ge h_5+1\)}
	In this case, we have \(h_1 + h_2 + h_3 + h_4\le (h-1)/2\). Without loss of generality, we assume \(h_1=\min\{h_1,\dots,h_4\}\). Note that \(e_1=I\) is incident with \(h_1\). We have 
	\[
	    h_1\le \frac{h-1}{8}
	    \quad\text{and}\quad
	    \max\{h_1,\dots,h_4\}\le \frac{h-7}{2}.
	\]
	\begin{enumerate}[label=\textbf{Case~\thecase}-(\arabic*):, wide,nosep]
		\item \(f_1= \min\{c_i\}\neq c_2\).
		    Let \(e_2\defeq h_1\). Then $e_1=I$ and $e_2=h_1$ are incident in $H$ while $f_1$ and $f_2$ are incident in $G_0$.
		    Hence \eqref{eq: uniquemap} yields a unique embedding. These embedding orders are proper because they satisfy \eqref{eq: proper} as
			\begin{align*}
				f_2&\ge \frac{h-1}{2}\ge 1+\frac{h-1}{8}\ge 1+h_1=1+e_2
				\quad\text{and} \\
				f_3&\ge \frac{3h-3}{4}\ge 1+\frac{h-1}{8}+\frac{h-7}{2} \ge 1+h_1+\max\{h_1,\dots,h_4\}\ge 1+e_2+e_3.
			\end{align*}
		
		\item \(f_1= \min\{c_i\}=c_2\) and \(\mins{2}\{c_i\}\ge 1+h_1+h_5\).
		    As \(f_1=c_2\) and \(f_2=b_1\) are not incident, we choose $e_3$ first instead of $e_2$. Also, let \(e_3\defeq h_1\). As $e_1=I$ and $e_3=h_1$ are incident in $H$ while $f_1$ and $f_3$ are incident in $G$, \eqref{eq: uniquemap} yields a unique embedding, and we have $e_2 = h_5$ in this unique ordering.
		    This uniquely determines \((e_i)_i\), which is proper as
    		\begin{align*}
    			f_2&\ge 1+h_5=1+e_2 \quad\text{and}\\
    			f_3&\ge 1+h_5+h_1=1+e_2+e_3.
    		\end{align*}
    		
		\item \(f_1=\min\{c_i\}=c_2\) and \(\mins{2}\{c_i\}\le h_1+h_5\).
    		This is the only case in which the prescribed \((f_i)_{i=1}^6\) is not used. Suppose \(c_2\le c_1\le c_3\) as the other case is similar. From
    		\[
    			\frac{3h-3}{4}\le c_1\le h_1+h_5\le \frac{h-1}{8}+h_5,
    		\]
    		we see \(h_5\ge (5h-5)/8\). As \(h_1+\dots+h_5=h-1\), this yields new bounds
    		\[
    			h_1= \min\{h_1,\dots, h_4\} \le\frac{3h-3}{32},\qquad
    			\max\{h_1,\dots,h_4\}\le\frac{3h-27}{8}.
    		\]
    		In addition, from \(c_1+c_2+c_3\ge \max\{3h-3,2k\}\) we have
    		\[
    		    c_1+c_2
    		    \ge\begin{cases*}
    		        (3h-3)-\frac{3h-4}{2}  \ge \frac{3h-3}{2} & if \(h\le k\le \frac{3h-4}{2}\), \\
    		        2k-k  \ge \frac{3h-3}{2} & if \(k\ge\frac{3h-3}{2}\).
    		    \end{cases*}
    		\]
    		With this and \(c_1\le h_1+h_5\leq  h-(1+h_2+h_3+h_4)\le h-4\), it follows \(c_2\ge (h+5)/2\). Set \((e_i)_i\defeq(I,h_1,h_4,h_2,h_5,h_3)\) and \((f_i)_i\defeq(c_1,b_1,c_2,c_3,b_2,a)\). This is proper as the inequalities in \eqref{eq: granted_order} holds along with the following ones:
    		\begin{align*}
    			f_2&\ge \frac{h-1}{2}\ge 1+\frac{3h-3}{32}\ge 1+h_1=1+e_2, \\
    			f_3&\ge \frac{h+5}{2}\ge 1+\frac{3h-3}{32}+\frac{3h-27}{8}
    				\ge 1+h_1+\max\{h_1,\dots,h_4\}\ge 1+e_2+e_3.
    		\end{align*}
	\end{enumerate}
	
	The remaining case is when \(h_5\ge (h-1)/2\) and \(b_1\le h_5\), which is equivalent to the exceptional case \(h_5\ge b_1\) in the statement.
\end{proof}

In order to prove \Cref{thm:ex-gph_4-vtx-4-cc-mgph}, we first prove the following stronger lemma  which implies the theorem straightforwardly. We will use this lemma later in the proof of \Cref{thm:ex-gph_stab_r-cc-mgph_k-small}.

\begin{lemma}\label{lem:ex-gph_4-vtx-4-cc-mgph_weak-deg-bd}
Suppose \(k\ge h\) and \(0<\frac{1}{n}\ll\delta\ll\frac{1}{k}<1\). 
    Let \(H\) be an \(4\)-vertex \(4\)-color-critical multigraph with \(h\) edges, and \(k^*\defeq k^*(H)=\frac{3}{2}(h-1)\). Let \(G\) be a simply \(k\)-colored multicolored-\(H\)-free multigraph of order \(n\) such that
    \[
        \delta(G)\ge
        \begin{cases*}
            (1-\delta)(h-1)(n-1) & if \(h\le k<k^*\), \\
            (1-\delta)kd_3(n) & if \(k\ge k^*\).
        \end{cases*}
    \]
    \begin{enumerate}
        \item If \(h\le k<k^*\), then \(w(e)\le h-1\) for all \(e\in P(G)\).
        \item If \(k\ge k^*\) and \(G\) has an edge with multiplicity at least \(h\), then \(G\) is \(3\)-partite.
    \end{enumerate}
\end{lemma}

\begin{proof}[Proof of \Cref{thm:ex-gph_4-vtx-4-cc-mgph} using \Cref{lem:ex-gph_4-vtx-4-cc-mgph_weak-deg-bd}]
    Let \(H\) be a \(4\)-vertex \(4\)-color-critical multigraph with \(h\) edges, and \(k^*\defeq k^*(H)=\frac{3}{2}(h-1)\).
    
    First consider when \(h\le k<k^*\). Let \(G\) be a simply \(k\)-colored multigraph of order \(n\) not containing a multicolored copy of \(H\), such that \(e(G)\ge (h-1)\binom{n}{2}\). By \Cref{prop:min-deg-condition} we may assume \(\delta(G)\ge (h-1)(n-1)\). Then \Cref{lem:ex-gph_4-vtx-4-cc-mgph_weak-deg-bd} applies so that every edge in \(G\) has multiplicity at most \(h-1\), which gives the desired result.
    
    Next, consider when \(k\ge k^*\). By \Cref{prop:k-large_enough-to-check_k=k*} it is enough to prove for \(k=\lceil k^* \rceil\). Let \(G\) be a simply \(k\)-colored multigraph of order \(n\) not containing a multicolored copy of \(H\), such that \(e(G)\ge k t_3(n)\). By \Cref{prop:min-deg-condition} assume \(\delta(G)\ge kd_3(n)\). There is an edge with multiplicity at least \(h\), since otherwise \(e(G)\le (h-1)\binom{n}{2}< k t_3(n)\), a contradiction. Then \Cref{lem:ex-gph_4-vtx-4-cc-mgph_weak-deg-bd} applies so that \(G\) is \(3\)-partite, which gives the desired result.
\end{proof}

To prove \Cref{lem:ex-gph_4-vtx-4-cc-mgph_weak-deg-bd}, we collect the following lemma. It implies that, if there is one edge in $G$ of large multiplicity, then we can find a $H$-friendly subgraph for a $4$-vertex $4$-color-critical multigraph $H$ within a multigraph $G$ with large minimum degree.
    
\begin{lemma}\label{lem:ex-gph_4-vtx-4-cc-mgph_weak-deg-bd_skeleton-vtxs}
 Suppose \(k\ge h\) and \(0<\frac{1}{n}\ll\delta\ll\frac{1}{k}<1\), and \(k^*\defeq k^*(4,h)=\frac{3}{2}(h-1)\).
    Let \(H\) be a \(4\)-vertex \(4\)-color-critical multigraph with \(h\) edges and \(G\) be a simply \(k\)-colored multicolored-\(H\)-free multigraph of order \(n\) with
    \[
        \delta(G)\ge
        \begin{cases*}
            (1-\delta)(h-1)(n-1) & if \(h\le k<k^*\), \\
            (1-\delta)kd_3(n) & if \(k\ge k^*\).
        \end{cases*}
    \]
    
    Suppose there is an edge \(v_1v_2\) of multiplicity at least \(h\). Then there is \(v_3\in V(G)-\{v_1,v_2\}\) such that \(G[\{v_1,v_2,v_3\}]\) is \(H\)-friendly.
\end{lemma}

\begin{proof}
    Label the edge multiplicities of \(H\) as in \Cref{fig:(clm:(case:(lem:ex-gph_4-vtx-4-cc-mgph_weak-deg-bd)-pf_k-small)_S-indep)_labeling_gphs}.
    Let $k'\defeq \max\{k^*, k\}$.

	Suppose first that \(h_5< (h-1)/2\).
	As $(1-\delta)k^* d_3(n) \geq (1-\delta)(h-1)(n-1)$, using \Cref{prop:min-deg=>vtx_large-d_T(v)}, choose \(v_3\in V-\{v_1,v_2\}\) such that \(w(v_1v_3)+w(v_2v_3)\ge 2(h-1).\)
    We claim that $G[\{v_1,v_2,v_3\}]$ is \(H\)-friendly.
	Indeed, if there exists $v\notin \{v_1,v_2,v_3\}$ such that 
	$w(v_1v)+ w(v_2v)+ w(v_3v) \geq \max\{2k, 3(h-1)\}$ and $\min\{ w(v_1v), w(v_2v),w(v_3v)\}\geq 1$, then with these choices, it is straightforward that the assumptions on  \Cref{lem:4-vtx-4-cc-mgph_small-case} are met with $G[\{v_1,v_2,v_3,v\}],v_1,v_2,v_3,v$ playing the roles of $G_0,v_1,v_2,v_3,v_4$ in \Cref{lem:4-vtx-4-cc-mgph_small-case}.
	We can apply \Cref{lem:4-vtx-4-cc-mgph_small-case} to conclude that the induced subgraph \(G[\{v_1,v_2,v_3,v\}]\) contains a multicolored copy of \(H\). Hence $G[\{v_1,v_2,v_3\}]$ is \(H\)-friendly.

    Next, suppose \(h_5\ge (h-1)/2\).
	Define a set
	\[
		S\defeq
		\biggl\{v\in V-\{v_1,v_2\} : w(vv_1)+w(vv_2)\ge k'+\frac{h-h_5+1}{2}\biggr\}.
	\]
    
	\begin{claim}\label{clm:(case:(lem:ex-gph_4-vtx-4-cc-mgph_weak-deg-bd)-pf_k-small)_S-indep}
		\(S\) is independent.
	\end{claim}
	\begin{claimproof}
		Suppose to the contrary that there is an edge \(v_3v_4\) in \(S\). Let \(G_0\defeq G[\{v_1,\dots,v_4\}]\), and label the edge multiplicities as in \Cref{fig:(clm:(case:(lem:ex-gph_4-vtx-4-cc-mgph_weak-deg-bd)-pf_k-small)_S-indep)_labeling_gphs}. Without loss of generality assume \(m_1\le m_2\). As in the proof of \Cref{lem:4-vtx-4-cc-mgph_small-case}, for each edge, we use the label referring its multiplicity to refer the edge itself, and we denote \(I\defeq x_1x_2\). We claim \(G_0\) contains a multicolored \(H\), so we determine the proper edge embedding orders \((e_i)_{i=1}^6\) and \((f_i)_{i=1}^6\) of \(P(H)\) and \(P(G_0)\), respectively. We assign: \(e_1=I\), \(f_1=v_3v_4\) and \(e_6=h_5\), \(f_6=v_1v_2\).
		\begin{figure}[!htbp]
			\centering
			\begin{tikzpicture}[
				x=3cm, y=3cm,
				every node/.style={draw, circle, fill=white},
				labelnode/.style={draw=none, rectangle},
				]
				\draw (0,0) node(v2){\(v_2\)}
				-- node[labelnode]{\(\ge h\)} (0,1) node(v1){\(v_1\)}
				-- node[labelnode]{\(m_1\)} (1,1) node(v3){\(v_3\)}
				-- node[labelnode]{\(\ge 1\)} (1,0) node(v4){\(v_4\)}
				-- node[labelnode]{\(m_4\)} (v2);
				
				\draw (v2) -- node[labelnode,near end]{\(m_2\)} (v3)
				(v1) -- node[labelnode,near end]{\(m_3\)} (v4);
				
    			\node[labelnode] at (0.5,-0.25) {$G_0$};
			\end{tikzpicture}
			\qquad
    		\begin{tikzpicture}[
				x=3cm, y=3cm,
				every node/.style={draw, circle, fill=white},
				labelnode/.style={draw=none, rectangle},
				]
				\draw (0,0) node(x2){\(x_2\)}
				-- node[labelnode]{\(h_5\)} (0,1) node(x1){\(x_1\)}
				-- node[labelnode]{\(h_1\)} (1,1) node(x3){\(x_3\)}
				-- node[labelnode]{\(I=1\)} (1,0) node(x4){\(x_4\)}
				-- node[labelnode]{\(h_4\)} (x2);
				
				\draw (x2) -- node[labelnode,near end]{\(h_2\)} (x3)
				(x1) -- node[labelnode,near end]{\(h_3\)} (x4);
				
    			\node[labelnode] at (0.5,-0.25) {\(H\)};
			\end{tikzpicture}
			\caption{The labelings of \(G_0\) and \(H\).}
			\label{fig:(clm:(case:(lem:ex-gph_4-vtx-4-cc-mgph_weak-deg-bd)-pf_k-small)_S-indep)_labeling_gphs}
		\end{figure}
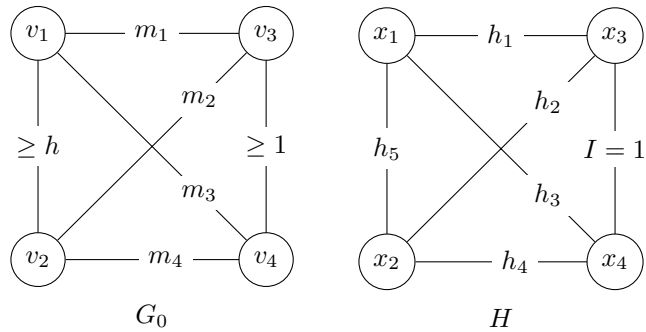
		
		Suppose \(m_4\le m_3\). Then as \(m_1\le m_2\), we have \(m_1+m_4\le m_2+m_3\),
		\begin{align*}
			m_1,m_4&\ge \biggl(k'+\frac{h-h_5+1}{2}\biggr) - k \ge \frac{h-h_5+1}{2},\\
			m_2,m_3&\ge \frac{1}{2} \biggl( \frac{3}{2}(h-1)+\frac{h-h_5+1}{2} \biggr)
				= h-\frac{h_5}{4}-\frac{1}{2}
				\geq h-h_5
				= h_1+h_2+h_3+h_4+1,
		\end{align*}
		and \(\min\{h_1+h_4,h_2+h_3\}\le (h-h_5-1)/2\). Without loss of generality, suppose \(h_1+h_4\le h_2+h_3\). Then set \((e_2,\dots,e_5)=(h_1,h_4,h_2,h_3)\) and \((f_2,\dots,f_5)=(m_1,m_4,m_2,m_3)\). It is straightforward to check that these embedding orders are consistent with respect to the vertex-edge incidence relations and they satisfy \eqref{eq:proper-edge-emb-orders_condition}. The case when \(m_4>m_3\) is similar, so we omit it. Therefore \(G_0\) contains a multicolored \(H\), a contradiction.
	\end{claimproof}
    
    We claim that there exists a vertex \(v_3\in S\) such that \(w(v_1v_3)+w(v_2v_3)\geq k'+h_5+1.\)
    Suppose not. Then we have \(w(v_1v)+w(v_2v) < k'+h_5+1\) for all \(v\in S\). Then, by the minimum degree condition of $G$ and the fact that $d_3(n) \ge \frac{2}{3}(n-1)$, we have that $\delta(G) \ge (1-\delta) k' \cdot \frac{2}{3}(n-1)$. Thus,
    \begin{align*}
	    2(1-\delta) k' \cdot \frac{2}{3}(n-1) &\le d(v_1)+d(v_2)
	    \\&= e(\{v_1,v_2\},S) + e\bigl(\{v_1,v_2\},V-(S\cup\{v_1,v_2\})\bigr) + 2w(v_1v_2)
		\\&\le \left(k'+h_5+\frac{1}{2}\right)|S|+\biggl(k'+\frac{h-h_5}{2}\biggr)(n-|S|-2) + 2k.
	\end{align*}
	Here, the final inequality holds by the definition of \(S\). Using $h_5 \le h - 5$, we have
	\begin{align*}
		|S|
		&\ge \frac{(2k'-3h+3h_5-8\delta k')n-12k+4(1+2\delta)k'+6(h-h_5)}{3(3h_5-h+1)}
		\\&> \frac{(2k'-3h+3h_5-1)n}{3(3h_5-h+1)}
		= \frac{n}{3} + \frac{(2k'-2h-2)n}{3(3h_5-h+1)} \\
		&\geq \frac{n}{3} + \frac{(2k'-2h-2)n}{3(2h-14)}.
	\end{align*}
	Note that we have \((h+1)/2 \ge h-h_5 \ge h_1+h_2+h_3+h_4+1 \ge 5\), so \(h\ge 11\). As \(k'\ge\frac{3}{2}(h-1)\), we conclude that \(|S|\ge \frac{n}{3} + \frac{(h-5)n}{3(2h-14)} > \frac{n}{2}\). This fact together with \Cref{clm:(case:(lem:ex-gph_4-vtx-4-cc-mgph_weak-deg-bd)-pf_k-small)_S-indep} implies that for any \(v\in S\),
	\[
		d(v)\le k|V-S| < \frac{kn}{2} < (1-\delta)(h-1)(n-1) \le (1-\delta)k^*d_3(n),
	\]
	a contradiction to the minimum degree condition. Thus, there exists $v_3 \in S$ such that $w(v_1v_3)+w(v_2v_3)\ge k+h_5+1$. Fix such a vertex $v_3$.
	
    We claim that \(G[\{v_1,v_2,v_3\}]\) is \(H\)-friendly.
	Indeed, if there exists a vertex \(v\notin \{v_1,v_2,v_3\}\) such that \(w(v_1v)+w(v_2v)+w(v_3v)\ge \max\{2k, 3(h-1) \}\) and \(\min\{w(v_1v),w(v_2v),w(v_3v)\}\ge 1\), then \Cref{lem:4-vtx-4-cc-mgph_small-case} applies with \(\min\{w(v_1v_3),w(v_2v_3)\}\ge h_5+1\) to get a multicolored copy of \(H\) in \(G[v_1,v_2,v_3,v]\).
	
	Therefore, there exists \(v_3\in V-\{v_1,v_2\}\) making \(G[\{v_1,v_2,v_3\}]\) \(H\)-friendly.
\end{proof}

We finish this section by proving \Cref{lem:ex-gph_4-vtx-4-cc-mgph_weak-deg-bd} as follows.

\begin{proof}[Proof of \Cref{lem:ex-gph_4-vtx-4-cc-mgph_weak-deg-bd}]
    \case{\(h\le k<k^*\)}\label{case:(lem:ex-gph_4-vtx-4-cc-mgph_weak-deg-bd)-pf_k-small}
    Suppose to the contrary that there is an edge \(v_1v_2\) with multiplicity at least \(h\). 
    By \Cref{lem:ex-gph_4-vtx-4-cc-mgph_weak-deg-bd_skeleton-vtxs}, there exists vertices $v_1,v_2,v_3$ such that $G[\{v_1,v_2,v_3\}]$ is $H$-friendly.
    
    Using \Cref{prop:min-deg=>vtx_large-d_T(v)}, we can find a vertex $v\in V-\{v_1,v_2,v_3\}$ such that 
    \(w(v_1v)+ w(v_2v)+ w(v_3v)\geq 3(h-1)\).
    As $k<k^*= \frac{3}{2}(h-1)$, we have $3(h-1) = \max\{2k, 3(h-1)\}$.
    Moreover, as each of $w(v_iv)\leq k$, we have that for each $i\in [3]$, 
    \( w(v_iv) \geq 3(h-1) -2k \geq 1.\)
    As $G[\{v_1,v_2,v_3\}]$ is $H$-friendly, this implies that $G[\{v_1,v_2,v_3,v\}]$ contains a multicolored copy of $H$, a contradiction.
    Hence every edge of \(G\) has multiplicity at most \(h-1\), completing the proof.
	
	\case{\(k\ge k^*\) and \(G\) has an edge with multiplicity at least \(h\)}
    By applying \Cref{lem:ex-gph_4-vtx-4-cc-mgph_weak-deg-bd_skeleton-vtxs}, we obtain a vertex set $\{v_1,v_2,v_3\}$ such that  \(K\defeq G[\{v_1,v_2,v_3\}]\) is an \(H\)-friendly multigraph.
	By applying \Cref{lem:H-friendly-submgph=>(r-1)-partite} with \(r=4\) and \(K\), we conclude that \(G\) is \(3\)-partite. This finishes the proof.
\end{proof}

\section{\texorpdfstring{\(r\)}{r}-vertex multigraphs in \texorpdfstring{\(\calF_r\)}{F\textrinferior} are good}
\label{sec:r-vtx-mgph_in_Fr_good}

Here we prove that for \(r\ge5\), the \(r\)-vertex multigraphs in \(\calF_r\) are good.

\begin{theorem}\label{thm:ex-gph_r-vtx-mgph_in_Fr}
    For \(r\ge 5\), the \(r\)-vertex multigraphs in \(\calF_r\) are good.
\end{theorem}
    
    

We again prove a stronger lemma which implies \Cref{thm:ex-gph_r-vtx-mgph_in_Fr} straightforwardly as in the proof of \Cref{thm:ex-gph_4-vtx-4-cc-mgph}.
Again, this stronger lemma will be useful later when we prove \Cref{thm:ex-gph_stab_r-cc-mgph_k-small}.

\begin{lemma}\label{lem:ex-gph_r-vtx-r-cc-mgph_in_Fr_weak-deg-bd}
Suppose \(k\ge h\) and \(0\ll\frac{1}{n}\ll\delta\ll\frac{1}{k}<1\).
    Let \(H\) be an \(r\)-vertex multigraph in \(\calF_r\) with \(h\) edges and \(r\ge 5\), and \(k^*\defeq k^*(H)\).  Let \(G\) be a simply \(k\)-colored multicolored-\(H\)-free multigraph of order \(n\) such that
    \[
        \delta(G)\ge
        \begin{cases*}
            (1-\delta)(h-1)(n-1) & if \(h\le k<k^*\), \\
            (1-\delta)kd_{r-1}(n) & if \(k\ge k^*\).
        \end{cases*}
    \]
    \begin{enumerate}
        \item If \(h\le k<k^*\), then \(w(e)\le h-1\) for all \(e\in P(G)\).
        \item If \(k\ge k^*\) and \(G\) has an edge with multiplicity at least \(h\), then \(G\) is (\(r-1\))-partite.
    \end{enumerate}
\end{lemma}
\begin{proof}
    We use the following lemma which allows to find an \(H\)-friendly submultigraph of \(G\). The proof works only when the edge multiplicities of \(H\) are bounded, which is the reason why we focus on \(\calF_r\).
    
    \begin{lemma}\label{lem:r-vtx-Fr_embedding}
        Let \(H\) be an \(r\)-vertex multigraph in \(\calF_r\) with \(h\) edges, and let \(k\ge h\) and \(k^*\defeq k^*(H)\). Let \(K\) be a simply \(k\)-colored multigraph on \(\{v_1,\dots,v_{r-1}\}\) such that \(w(v_1v_2)\ge h\) and for each \(3\le j\le r-1\),
        \[
            \sum_{i=1}^{j-1} w(v_iv_j)\geq (j-1)\max\biggl\{(h-1),\frac{r-2}{r-1}k\biggr\}.
        \]
        Then \(K\) is \(H\)-friendly.
    \end{lemma}
    
    First let \(h\le k<k^*\). Suppose to the contrary that there is an edge \(v_1v_2\) with multiplicity at least \(h\). By \Cref{prop:min-deg=>vtx_large-d_T(v)}, we can find \(v_3,\dots,v_r\) such that \(\sum_{i=1}^{j-1} e(v_iv_j)\ge (j-1)(h-1)\) for each \(3\le j\le r\). By \Cref{lem:r-vtx-Fr_embedding} we see \(G[\{v_1,\dots,v_{r-1}\}]\) is \(H\)-friendly, so \(G[\{v_1,\dots,v_{r-1},v_r\}]\) contains a multicolored \(H\), a contradiction.
    
    Next let \(k\ge k^*\) and \(v_1v_2\) be an edge with multiplicity at least \(h\). By \Cref{prop:min-deg=>vtx_large-d_T(v)}, we can find \(v_3,\dots,v_{r-1}\) such that \(\sum_{i=1}^{j-1} e(v_iv_j)\ge (j-1)\frac{r-2}{r-1}k\) for each \(3\le j\le r-1\). \Cref{lem:r-vtx-Fr_embedding} says that \(G[\{v_1,\dots,v_{r-1}\}]\) is \(H\)-friendly, thus \Cref{lem:H-friendly-submgph=>(r-1)-partite} applies so that \(G\) is (\(r-1\))-partite.
\end{proof}

We finish this section with the following proof of \Cref{lem:r-vtx-Fr_embedding}.
\begin{proof}[Proof of \Cref{lem:r-vtx-Fr_embedding}]
    Let \(v_r\) be a vertex not in \(K\) such that \(\sum_{i=1}^{r-1}w(v_iv_r)\ge (r-2)\max\{k,k^*\}\) and \(\min\{w(v_iv_r)\}\ge1\), and let \(G_0\defeq G[V(K)\cup \{v\}]\). We need to show \(G_0\) contains a multicolored \(H\).
    For each \(2\le j\le r\), enumerate \(\{v_iv_j:1\le i\le j-1\}\) into an ascending order \(e_{j,1}, \dots,e_{j,j-1}\)  of multiplicities. Since \(w(e)\le k\) for each \(e\in P(G)\), we see that for all \(2\le i\le r\) and \(1\le i\le j-1\),
    \begin{align*}
        w(e_{j,i})
        &\ge
        \frac{1}{i}\left( \sum_{\ell =1}^{j-1}w(v_{\ell}v_j) - \sum_{\ell=i+1}^{j-1} w(e_{j,\ell})\right) \\
        &\ge
        \frac{1}{i}\left(  (j-1)\max\left\{h-1, \frac{r-2}{r-1}k \right\}  - (j-1-i)k\right)
        \\&\ge
        \frac{1}{i}\left( (j-1)\cdot \max\left\{ h-1, \frac{r-2}{r-1}k \right\} -(j-1-i)\cdot \frac{r-1}{r-2}\max\left\{h-1, \frac{r-2}{r-1}k \right\}\right)
        \\&\ge
        \frac{i(r-1)-(j-1)}{i(r-2)}\max\left\{h-1, \frac{r-2}{r-1}k \right\}\\
        &\geq \frac{i(r-1)-(j-1)}{i(r-2)}(h-1)
        \eqdef b_{j,i}.
    \end{align*}
    We have \(w(e_{2,1})=w(v_1v_2)\ge h\) by the assumption. If \(h\le k<k^*\), then the third inequality is strict, so the above gives \(w(e_{r,1})>b_{r,1}=0\), i.e., \(w(e_{r,1})\ge 1\); if \(k\ge k^*\) we have \(w(e_{r,1})\ge1\) by the assumption.
    We embed \(H\) in \(G_0\) so that a critical edge of \(H\) is embedded in \(e_{r,1}\) and the vertex-edge incidence relation is preserved, with the edge embedding order of \(P(H)\) given as
    \begin{align*}
        &e_{r,1},\,e_{r-1,1},\,\dots,\,e_{4,1},\,e_{3,1}, \\
        &e_{r,2},\,e_{r-1,2},\,\dots,\,e_{4,2}, \\
        &\quad\vdots \\
        &e_{r,r-3},\,e_{r-1,r-3}, \\
        &e_{r,r-2}, \\
        &e_{r,r-1},\,e_{r-1,r-2},\,\dots,\,e_{2,1}.
    \end{align*}
    Since \(w(e_{j,j-1})\ge b_{j,j-1}=h-1\) for \(3\le j\le r\) and \(w(e_{2,1})\ge h\), this is a proper embedding order if the \(m\)-th edge for \(2\le m\le\binom{r}{2}-(r-1)\), has multiplicity at least \((m-1)\cdot \alpha_r(h-1)+1\) where \(\alpha_r\defeq\frac{2+2/r^2}{(r-1)(r-2)}\) (recall that \(w(e)\le\alpha_r(h-1)\) for all edges \(e\) of \(H\) by the definition of \(\calF_r\)). Since each \((b_{j,i})_{j=r}^{i+2}\) for a fixed \(i\) is an arithmetic progression, it suffices to check this for the edges 
    \[e_{r,i} \text{ and } e_{i+2,i} \text{ for each } 1\le i\le r-2, \text{ as well as } e_{r-1,1}.\] 

    We start with the edges $e_{r,i}$. As we embed a critical edge of $H$ in $e_{r,1}$, we only need to consider $2 \le i \le r-2$. For \(2\le i \le r-2\), the edge \(e_{r,i}\) is the \(1+\sum_{\ell=1}^{i-1}(r-(\ell+1))=(i-1)(r-1)-\frac{i(i-1)}{2}+1\eqdef m_{1,i}\)-th element in the list. Then we need to show
    \[
        F_{1,i}\defeq b_{r,i}-(m_{1,i}-1)\alpha_r(h-1)> 0.
    \]
    Note that for \(2\le i\le r-2\), we have
    \[ 
        \frac{b_{r,i}}{h-1}
        = \frac{r-1}{r-2} - \frac{r-1}{i(r-2)}
        = \frac{i-1}{r-2}\cdot \frac{r-1}{i}
    \]
    and 
    \begin{align*}
    (m_{1,i}-1)\alpha_r &= (i-1)\left((r-1)-\frac{i}{2} \right) \frac{1}{r-2}\left(\frac{2}{r-1} +\frac{2}{r^2(r-1)}\right) \\
    &= \frac{i-1}{r-2}\left( 2 - \frac{i}{r-1} + \frac{2}{r^2} - \frac{i}{r^2(r-1)} \right).
    \end{align*}

    Hence, we have
    \begin{align*}
    \frac{r-2}{i-1}\cdot \frac{F_{1,i}}{(h-1)} &= \frac{r-2}{i-1}\left(\frac{b_{r,i}}{h-1} - (m_{1,i}-1)\alpha_r \right)
    \\&= 
    \left(\frac{r-1}{i} - 2 + \frac{i}{r-1}\right) - \left(\frac{2}{r^2}- \frac{i}{r^2(r-1)}\right) \\
    &=  \frac{(r-1-i)^2}{(r-1)i} + \frac{i^2-2(r-1)i}{r^2(r-1)i} \\
    &= 
    \frac{r^2(r-1-i)^2}{r^2(r-1)i} + \frac{(r-1-i)^2 - (r-1)^2}{r^2(r-1)i} 
     >0
    \end{align*}
    The final inequality holds as $r^2(r-1-i)^2\geq (r-1)^2$.
    
    
    Next, as the edge $e_{r,r-2}$ is already considered in the last case, we only need to deal with the edges \(e_{i+2,i}\) for $1 \le i \le r-3$. For $1 \le i \le r-3$, the edge \(e_{i+2,i}\) is the \(\sum_{\ell=1}^{i} (r-(\ell+1))=i(r-1)-\frac{i(i+1)}{2}\eqdef m_{2,i}\)-th element in the list. Then we need to show
    \[
        F_{2,i}\defeq b_{i+2,i}-(m_{2,i}-1)\alpha_r(h-1) > 0.
    \]
    We have
    \[\frac{b_{i+2,i}}{h-1} = \frac{i(r-1)-i-1}{i(r-2)} = \frac{i}{r-2} \left( \frac{r-2}{i}- \frac{1}{i^2}\right)\]
    and
    \begin{align*}
    (m_{2,i}-1)\alpha_r &=
    \left(i(r-1) - \frac{i(i+1)}{2} -1\right) \alpha_r\\
    &= \frac{i}{r-2}\left( (r-1) - \frac{i(i+1)+2}{2i}\right)\left(\frac{2}{r-1} + \frac{2}{r^2(r-1)} \right) \\
    &\leq 
    \frac{i}{r-2}\left( 2 - \frac{i(i+1)+2}{i(r-1)} + \frac{2}{r^2}\right) \\
    &= \frac{i}{r-2}\left( 2 -\frac{i+1}{r-1} - \frac{2}{i(r-1)} + \frac{2}{r^2}\right)
    \end{align*}
    Hence, we have
    \begin{align*}
    \frac{r-2}{i}\cdot \frac{F_{2,i}}{h-1}
    &= \frac{r-2}{i}\biggl( \frac{b_{i+2,i}}{h-1}-(m_{2,i}-1)\alpha_r \biggr)
    \\&\ge 
    \biggl(\frac{r-2}{i} - \frac{1}{i^2}\biggr) - \biggl(2 - \frac{i+1}{r-1} - \frac{2}{i(r-1)} + \frac{2}{r^2} \biggr)
    \\
    &> \left(\frac{r-2}{i}+ \frac{i+1}{r-1}-2\right) - \frac{1}{i^2} + \left(\frac{2}{i(r-1)} - \frac{2}{(r-1)^2}\right) \\
    &\ge \frac{(r-i-1)(r-i-2)}{i(r-1)} - \frac{1}{i^2} \\
    &= \frac{i(r-i-1)(r-i-2) - (r-1)}{i^2(r-1)} \ge 0.
    \end{align*}
    The last inequality holds as we have $i(r-i-1)(r-i-2) - (r-1)\ge 0$ because $1 \le i \le r-3$. 
    
    Finally, observe that $w(e_{r-1,1}) \ge \frac{h-1}{r-2} > \alpha_r(h-1)$. This finishes the verification that the edge embedding is indeed a proper one, which ends the proof of \Cref{lem:r-vtx-Fr_embedding}.
\end{proof}

\section{Stability for \texorpdfstring{\(r\)}{r}-vertex \texorpdfstring{\(r\)}{r}-color-critical multigraphs}
\label{sec:ex-gph_stab_good-mgph}

In previous two sections, we have determined \(\Ex_k(n,H)\) for all 
\(4\)-vertex \(4\)-color-critical multigraph or an \(r\)-vertex multigraph in \(\calF_r\) with \(r\ge 5\).

We now prove that for such a graph \(H\), if a simply \(k\)-colored multigraph \(G\) with no multicolored copy of \(H\) has close-to-maximum number of edges, then \(G\) is close to one of two natural extremal graphs.
We use the ideas in \cite[Lemma~2.3]{roberts2018stability} to prove this. For the convenience of writing, we use the following notation. For multigraphs \(G_1\) and \(G_2\) of the same order, define their \emph{symmetric difference}
\[
    |G_1\symmdiff G_2|\defeq \min_{\substack{G_2'\cong G_2 \\ V(G_2')=V(G_1)}} \sum_{e\in P(G_1)} |w_{G_1}(e)-w_{G_2'}(e)|
\]
as the minimum number of edges needed to be changed from \(G_1\) to make it isomorphic to \(G_2\).

\begin{theorem}\label{thm:ex-gph_stab_r-cc-mgph_k-small}
    Suppose \(h\le k\) and \(0<\frac{1}{n}<\eta\ll\varepsilon,\frac{1}{k}<1\).
    Let \(H\) be an \(r\)-vertex \(r\)-color-critical multigraph with \(h\) edges with \(r\geq 4\).
    Furthermore, if \(r\ge 5\), then assume \(H\in \calF_r\). Let \(G\) is a simply \(k\)-colored multicolored-\(H\)-free multigraph of order \(n\) such that
    \[
        e(G)\ge
        \begin{cases*}
            (h-1)\binom{n}{2} - \eta n^2 & if \(k<k^*\defeq k^*(H)\), \\
            kT_{r-1}(n) - \eta n^2 & if \(k\ge k^*\).
        \end{cases*}
    \]
    Then we have:
    \begin{enumerate}
        \item if \(k=k^*\), then either \(|G\symmdiff (h-1)K_n|\le\varepsilon n^2\) or \(|G\symmdiff kT_{r-1}(n)|\le\varepsilon n^2\);
        \item if \(k\neq k^*\), then \(|G\symmdiff \Ex_k(n,H)|\le\varepsilon n^2\).
    \end{enumerate}
\end{theorem}

We remark that when $k=k^*$, the difference between the numbers of edges in $(h-1)K_n$ and $kT_{r-1}(n)$ is $o(n^2)$ (in fact, it is $O(n)$), thus in this case, we consider two possibilities in \Cref{thm:ex-gph_stab_r-cc-mgph_k-small}.


In order to prove this theorem, we will show that for those choices of $H$, the graph \(G\) in \Cref{thm:ex-gph_stab_r-cc-mgph_k-small} have a large minimum degree after deleting a small number of vertices.
Then we may apply \Cref{lem:ex-gph_4-vtx-4-cc-mgph_weak-deg-bd} or \Cref{lem:ex-gph_r-vtx-r-cc-mgph_in_Fr_weak-deg-bd} to the remaining graph with high minimum degree. 

We first collect the following useful proposition.

\begin{proposition}\label{prop: deletion J}
    Suppose $0< \frac{1}{n} \ll \delta < 1$.
    Let $G$ be an $n$-vertex multigraph with $d\binom{n}{2}$ edges. If $B\subseteq V(G)$ is a vertex set with $|B|=\frac{1}{2}\delta n$ and every $v\in B$ satisfies $d_G(v)< (1 -\delta)dn$, then $G-B$ has at least $\left(1+\frac{1}{2}\delta^2\right)d\binom{|G-B|}{2}$ edges.
\end{proposition}
\begin{proof}
    Consider $J\defeq G-B$, then we have
    \begin{align*}
		e(J)
		\ge e(G)-\sum_{v\in B} d(v)
	    &> d\binom{|G-B|+|B|}{2} - (dn-\delta dn)|B|
	    \\&= d\binom{|G-B|}{2} + d|G-B||B| + d\binom{|B|}{2} - d(|G-B|+|B|)|B| + \delta d n |B| 
	    \\&> d\binom{|G-B|}{2} + d|B|(\delta n- |B|) 
	    \\&= d\binom{|G-B|}{2} + \frac{1}{4}d\delta^2 n^2
	    > \biggl(1+\frac{1}{2}\delta^2\biggr)d\binom{|G-B|}{2}. \qedhere
    \end{align*}
\end{proof}

\begin{proof}[Proof of \Cref{thm:ex-gph_stab_r-cc-mgph_k-small}]
    \Cref{thm:ex-gph_4-vtx-4-cc-mgph,thm:ex-gph_r-vtx-mgph_in_Fr} show that \(H\) is good in either of the cases.
    Note that if $k\neq k^*$, then $|k-k^*|\geq \frac{1}{r-2}$ from the integrality of $k$ and the definition of $k^*$.
    Note that the the minimum degree $\delta(\Ex_k(n,H))$ is $\max\bigl\{\frac{r-2}{r-1}k,h-1\bigr\}(n-1)+O(1) = \max\bigl\{k,\frac{r-1}{r-2}(h-1)\bigr\}d_r(n)+O(1)$.
    
	Let \(\delta\) be such that \(0<\eta\ll\delta\ll\varepsilon,\frac{1}{k}<1\). We first show that there is a submultigraph of \(G\) of order at least \((1-\delta^{1/2})n\) with minimum degree at least \((1-\delta^{1/2})\max\left\{k,\frac{r-1}{r-2}(h-1)\right\}d_r(n)\).

	Let $d'>0$ be the number such that $G$ contains $d'\binom{n}{2}$ edges, then $d'\geq  \max\left\{\frac{r-2}{r-1}k, (h-1)\right\}-2\eta$.
	Define
	\[
		L\defeq\{v\in G : d(v) < (1 - \delta)d'n\}.
	\]
	
    If $|L|\geq \delta n/2$, then choose any \(B\subseteq L\) with \(|B|=\delta n/2\). 
    By \Cref{prop: deletion J}, we conclude that 
    the multigraph $G-B$ contains at least the following number of edges:
    \begin{align*} \left(1+\frac{1}{2}\delta^2 \right)d'\binom{|G-B|}{2} &>
    \left(\max\left\{\frac{r-2}{r-1}k, (h-1)\right\}-2\eta + \frac{1}{2}\delta^2\right)\binom{|G-B|}{2}  \\ &> \max\left\{k,\frac{r-1}{r-2}(h-1)\right\} t_{r-1}(|G-B|). 
    \end{align*}

	On the other hand, \(G-B\) does not contain a multicolored copy of \(H\) and \(|G-B|>n/2 > \eta^{-1}/2\) with \(\eta \ll 1/h\). Thus since \(H\) is good, we have that \(e(G-B)\le \max\{(h-1)\binom{|G-B|}{2},kt_{r-1}(|G-B|)\} \le \max\{k,\frac{r-1}{r-2}(h-1)\} t_{r-1}(|G-B|)\),
	a contradiction.
	Hence, we have $|L|<\delta n/2$.
	
	Now consider the graph \(J\defeq G-L\). Then \(\sum_{v\in L}d(v)\le kn|L|\), and
	\begin{equation}\label{eq:(thm:ex-gph_stab_r-cc-mgph_k-small)-pf_mindeg-J}
		\delta(J)\ge (1-\delta)\max\left\{k,\frac{r-1}{r-2}(h-1)\right\}d_r(n)-k|L|>(1-\delta^{1/2})\max\left\{k,\frac{r-1}{r-2}(h-1)\right\}d_r(n).
	\end{equation}
	By \Cref{lem:ex-gph_4-vtx-4-cc-mgph_weak-deg-bd} for \(r=4\) or \Cref{lem:ex-gph_r-vtx-r-cc-mgph_in_Fr_weak-deg-bd} for \(r\ge5\), we obtain that \(J\) is either a subgraph of $(h-1)K_{|J|}$ or an $(r-1)$-partite graph.
	
	First, suppose that \(J\) is a subgraph of \((h-1)K_{|J|}\). 
	If $k>k^*$, then 
	\[e(G)\geq \ex_k(n,H)-\eta n^2 \geq \left(k^*+ \frac{1}{r-2}\right)t_{r-1}(n)-\eta n^2
	> (h-1)\binom{n}{2},
	\]
	a contradiction. Hence we have $k\leq k^*$ and
	as we have \(e(J)\geq \delta(J)|J|/2 \geq (h-1)\binom{|J|}{2} - \delta^{1/3} n^2\) from \eqref{eq:(thm:ex-gph_stab_r-cc-mgph_k-small)-pf_mindeg-J}, we conclude that
		\begin{align*}
	    |G\symmdiff (h-1)K_n|
	    &\le  \sum_{e\in P(G)-P(J)}|(h-1)-w(e)| + \biggl( (h-1)\binom{|J|}{2} - e(J) \biggr)
	    \\&\le kn|L| + \biggl( (h-1)\binom{|J|}{2} - \biggl((h-1)\binom{|J|}{2} - \delta^{1/3} n \biggr) \biggr)
	    \\&\le 2\delta^{1/3} n^2 < \varepsilon n^2. 
	\end{align*}

    If \(J\) is not a subgraph of \((h-1)K_{|J|}\), then there is an edge \(v_1v_2\) in \(J\) with multiplicity at least \(h\). By \Cref{lem:ex-gph_4-vtx-4-cc-mgph_weak-deg-bd} for \(r=4\) or \Cref{lem:ex-gph_r-vtx-r-cc-mgph_in_Fr_weak-deg-bd} for \(r\ge 5\), we obtain that \(k\ge k^*\) and  \(J\) is an (\(r-1\))-partite graph with parts \(V_1,\dots,V_{r-1}\). If \(|V_i|-|V_j|\ge 2a\) for some \(i\neq j\), then \(e(J)\le t_{r-1}(|J|)-a^2\), thus this together with \eqref{eq:(thm:ex-gph_stab_r-cc-mgph_k-small)-pf_mindeg-J} implies that two classes can differ in size by at most \(2\delta^{1/5}n\). It follows that \(\bigl||V_i|-\frac{|J|}{r-1}\bigr|<2\delta^{1/5}n\) for each \(i\) and so by deleting at most \(kn(r-1)(2\delta^{1/5}n)<\delta^{1/6}n^2\) edges, we attain an (\(r-1\))-partite graph with class sizes equal to that of \(T_{r-1}(|J|)\), and the number of edges is at least \(kt_{r-1}(|J|)-2\delta^{1/6}n^2\). Therefore
	\[
	    |G\symmdiff kT_{r-1}(n)|\le |J\symmdiff kT_{r-1}(|J|)|+\sum_{v\in L}d(v) \le 3\delta^{1/6}n^2+kn|L| < \varepsilon n^2. \qedhere
	\]
\end{proof}

\section{Stability for \texorpdfstring{\(r\)}{r}-color-critical graphs on more than \texorpdfstring{\(r\)}{r} vertices}
\label{sec:mcol-reg-lem}

In this section, we prove the following lemma. This lemma states that for an \(r\)-color critical graph \(H\), if its color-reduced multigraph \(H_c\) satisfies stability, then \(H\) also satisfies stability. For our convenience, we assume that the coloring of the host multigraph $G$ is nested as discussed in the preliminaries.

\begin{lemma}\label{lem:ex-gph_stab_cc-gph_with_very-good-cred-mgph}
    Suppose \(k\ge h\), \(k^*\defeq k^*(r,h)\), and \(0<\frac{1}{n}\ll\eta\ll\mu,\frac{1}{k}\leq 1\).
    Let \(H\) be an \(r\)-color-critical graph with \(h\) edges and let \(H_c\) be the corresponding color-reduced multigraph. 
    If \(r\geq 5\), then assume \(H_c \in \mathcal{F}_r\).
    Suppose \(G\) is a simply \(k\)-nested-colored multicolored-\(H\)-free multigraph of order \(n\) such that
    \[
        e(G)\ge\begin{cases*}
            (h-1)\binom{n}{2}-\eta n^2 & if \(h\le k<k^*\), \\
            kt_{r-1}(n)-\eta n^2 & if \(k\ge k^*\).
        \end{cases*}
    \]
    Then we have either of the following:
    \begin{enumerate}
        \item If \(k<k^*\), then \(|G\symmdiff (h-1)K_n|\leq \mu n^2\);
        \item if \(k=k^*\), then either \(|G\symmdiff (h-1)K_n|\leq \mu n^2\) or \(|G\symmdiff kT_{r-1}(n)|\leq \mu n^2\);
        \item if \(k>k^*\), then \(|G\symmdiff kT_{r-1}(n)|\leq \mu n^2\).
    \end{enumerate}
\end{lemma}

We use the multicolor version of the Szemer\'{e}di regularity lemma to prove the above lemma. We start with relavant definitions. 
For a (multi)graph \(G\), define its \emph{edge density} as \(d(G)\defeq e(G)/|G|^2\). 

Let \(G\) be a simply \(k\)-colored multigraph. For disjoint nonempty vertex sets \(X\), \(Y\) and a color \(\rho\), we denote \(e_\rho(X,Y)\) to be the number of edges between \(X\) and \(Y\) with color \(\rho\). The \emph{\(\rho\)-density} of \((X,Y)\) is
\[
    d_\rho(X,Y)\defeq\frac{e_\rho(X,Y)}{|X||Y|}.
\]

For \(\varepsilon>0\) and a color \(\rho\), the pair \((X,Y)\) is \emph{\((\varepsilon;\rho)\)-regular} if for every \(X'\subseteq X\) and \(Y'\subseteq Y\) with \(|X'|\ge\varepsilon|X|\) and \(|Y'|\ge\varepsilon|Y|\), we have \(|d_\rho(X,Y)-d_\rho(X',Y')|\le\varepsilon\). Then \((X,Y)\) is \emph{\(\varepsilon\)-regular} if it is \((\varepsilon;\rho)\)-regular for all colors \(\rho\).
For \(\varepsilon,\gamma >0\) and a color \(\rho\), the pair \((X,Y)\) is \emph{\((\varepsilon,\gamma;\rho)\)-lower-regular} if for every \(X'\subseteq X\) and \(Y'\subseteq Y\) with \(|X'|\ge\varepsilon|X|\) and \(|Y'|\ge\varepsilon|Y|\), we have \(d_\rho(X',Y')\ge \gamma\).


A partition \(\mathcal{P}=(V_1,\dots,V_m)\) of \(V(G)\) is an \emph{\(\varepsilon\)-regular partition} of a simply \(k\)-colored multigraph \(G\) if
\begin{itemize}
    \item \(\bigl||V_i|-|V_j|\bigr|\le 1\) for all \(\{i,j\}\in\binom{[m]}{2}\);
    \item all but at most \(\varepsilon m^2\) of the pairs \((V_i,V_j)\), \(\{i,j\}\in\binom{[m]}{2}\) are \(\varepsilon\)-regular.
\end{itemize}

For \(\varepsilon,\gamma>0\) and
a given such an \(\varepsilon\)-regular partition \(\mathcal{P}\) of \(G\),
consider a simply \(k\)-colored multigraph \(R\) with colors \(\{R_1,\dots, R_k\}\) where the vertex set is \(\{v_1,\dots, v_m\}\) and for each \(i\neq j\in [m]\) and \(\rho \in [k]\), we have \(v_iv_j\in E(R_\rho)\) if and only if \((V_i,V_j)\) is \((\varepsilon,\gamma;\rho)\)-lower-regular. We call this multigraph \(R\) the \emph{\((\varepsilon,\gamma,\mathcal{P})\)-reduced multigraph} of \(G\). Consider a simply \(k\)-colored multigraph \(G^{\mathcal{P}}=G^{\mathcal{P}}(\varepsilon,\gamma)\) with vertex set \(V(G)\) where for all \(i\neq j\in [m]\) and \(\rho\in [k]\) the bipartite graph \(G^{\mathcal{P}}_\rho[V_i,V_j]\) is a complete bipartite graph if \(v_iv_j\in E(R_\rho)\), and an empty bipartite graph if \(v_iv_j\notin E(R_\rho)\).

If the simple \(k\)-coloring of \(G\) is nested, then for any \(\varepsilon,\gamma>0\) and an \(\varepsilon\)-regular partition \(\mathcal{P}\), the definition of lower-regularity ensures that \(R\) and \(G^{\mathcal{P}}\) is also nested.

The following is easily derived from the proof outline of \cite[Theorem~1.18 (Many-Color Regularity Lemma)]{komlos1996szemeredi}. In fact the setting there is about a \(k\)-edge-colored simple graph, but the idea of summing the indices for each color is valid in our setting. The moreover part of the statement can be also derived by a routine computation.

\begin{theorem}[Multicolor regularity lemma]\label{regularity lemma}
    For any \(\varepsilon>0\) and integers \(k,M_0\ge 1\), there exists \(M\) such that every simply \(k\)-colored multigraph \(G\) with \(n\ge M\) vertices admits an \(\varepsilon\)-regular partition \(\mathcal{P}=\{V_1,\dots, V_m\}\) with \(M_0\le m\le M\). Moreover, for \(\gamma>0\), the \((\varepsilon,\gamma,\mathcal{P})\)-reduced multigraph \(R\) of \(G\) satisfies \(d(R)\geq d(G)-2(\varepsilon+\gamma)\).
\end{theorem}

On the other hand, as all but at most \(\varepsilon m^2\) pairs \((V_i,V_j)\) are \(\varepsilon\)-regular, one can easily see the following holds.
\begin{equation}\label{GP close}
\begin{minipage}[c]{0.9\textwidth}
    There exists a set \(E\subseteq E(G)\) of at most \( \bigl(\varepsilon+\frac{1}{m}+\gamma\bigr) n^2\) edges of \(G\) such that \(G-E \subseteq G^{\mathcal{P}}\).
\end{minipage}
\end{equation}

One advantage of the regularity lemma is that it is useful to prove the existence of certain subgraph of \(G\).
A slight modification of the proof of \cite[Theorem~2.1 (Key Lemma)]{komlos1996szemeredi} gives the following.

\begin{theorem}[Multicolor embedding lemma]
    Suppose \(0< \frac{1}{n} \ll \varepsilon \ll \gamma, \frac{1}{h} \le 1\). Let \(H\) be a multigraph on \(r\) vertices and \(h\) edges.
    Let \(G\) be a simply nestedly \(k\)-colored multigraph, and \(\mathcal{P}\) be its \(\varepsilon\)-regular partition.
    If \(G^\mathcal{P}(\varepsilon,\gamma)\) contains a multicolored copy of \(H\), then \(G\) contains a multicolored copy of \(H\).
\end{theorem}

Note that the above theorem is only true when the coloring of $G$ is nested which can be assumed as mentioned in the preliminaries. Otherwise, all edges of \(G\) might have multiplicity \(1\) while \(H\) has an edge of larger multiplicity.

\begin{proof}[Proof of \Cref{lem:ex-gph_stab_cc-gph_with_very-good-cred-mgph}]
    Let \(0<\frac{1}{n}\ll\frac{1}{m_0}\ll\varepsilon\ll\gamma\ll\eta\ll\eta'\ll\mu,\frac{1}{k}\), and for \(s\ge1\), set
    \[
        G'(s)\defeq\begin{cases*}
            (h-1)K_s & if \(h\le k<k^*\), \\
            kT_{r-1}(s) & if \(k\ge k^*\).
        \end{cases*}
    \]
    We show that if \(e(G)\ge e(G'(n))-\eta n^2\), then either of the following holds:
    \begin{equation}\label{eq:cond_very-goodness}
    \begin{alignedat}{4}
        &|G\symmdiff G'(n)|<\mu n^2 &&\qquad\text{if } k\neq k^*, \\
        &|G\symmdiff (h-1)K_n|<\mu n^2 \quad\text{or}\quad |G\symmdiff kT_{r-1}(n)|<\mu n^2 &&\qquad\text{if } k=k^*.
    \end{alignedat}
    \end{equation}
    
    
    Apply \Cref{regularity lemma} to \(G\) with the constants \(\varepsilon,k,1/\varepsilon,m_0\) playing the roles of \(\varepsilon,k,M_0,M\) to obtain an \(\varepsilon\)-regular partition \(\mathcal{P}= \{V_1,\dots, V_m\}\) with \(1/\varepsilon\leq m\leq m_0\).
    
    Let \(R\) be an \((\varepsilon,\gamma,\mathcal{P})\)-reduced multigraph of \(G\). Then both \(R\) and \(G^\mathcal{P}\) have a nested simple \(k\)-coloring.
    Then \(R\) is multicolored-\(H_c\)-free, since otherwise \(G^{\mathcal{P}}\) contains a multicolored \(H\), and so does \(G\) by the embedding lemma. 
    
    By \eqref{GP close}, there exists a set \(E\subseteq E(G)\) of edges of \(G\) with \(|E|\leq (\varepsilon + 1/m + \gamma)n^2 \leq \eta n^2\) such that \(G-E\subseteq G^{\mathcal{P}}\).
    Also, by the moreover part of \Cref{regularity lemma}, we have \(e(R)\ge e(G'(m))-2\eta m^2\). 
    By \Cref{thm:ex-gph_stab_r-cc-mgph_k-small}, we have the following: 
    if \(k\neq k^*\), then  \(|R\symmdiff G'(m)|<\eta' m^2\). Otherwise, if \(k=k^*\), then \(|R\symmdiff (h-1)K_m|<\eta' m^2\) or \(|R\symmdiff kT_{r-1}(m)|<\eta' m^2\). 
    
    If \(|R\symmdiff (h-1)K_m|<\eta' m^2\), then \[ |G^{\mathcal{P}} \symmdiff (h-1)K_n| < 2\eta' n^2.\]
    This implies \(e(G^{\mathcal{P}}) \leq (h-1)\binom{n}{2} + 2\eta' n^2 \leq e(G) + 3\eta' n^2,\) thus 
    \( | G\symmdiff G^{\mathcal{P}}| \leq |E| + |E| + 3\eta'n^2 \leq 5\eta' n^2.\)
    Hence, we have
    \[ |G\symmdiff (h-1)K_n| \leq |G\symmdiff G^{\mathcal{P}}|+ |G^{\mathcal{P}} \symmdiff (h-1)K_n| \leq 8\eta' n^2 \leq \mu n^2.\]
    
    If \(|R\symmdiff kT_{r-1}(m)|<\eta' m^2\), then \[ |G^{\mathcal{P}} \symmdiff kT_{r-1}(n)| < 2\eta' n^2 + \frac{(r-1)n^2}{m} \leq 3\eta' n^2.\]
    This implies \(e(G^{\mathcal{P}}) \leq k t_{r-1}(n) + 3\eta' n^2 \leq e(G) + 4\eta' n^2,\) thus 
    \( | G\symmdiff G^{\mathcal{P}}| \leq |E| + |E| + 4\eta'n^2 \leq 6\eta' n^2.\)
    Hence, we have
    \[ |G\symmdiff kT_{r-1}(n)| \leq |G\symmdiff G^{\mathcal{P}}|+ |G^{\mathcal{P}} \symmdiff kT_{r-1}(n)| \leq 10\eta' n^2 \leq \mu n^2.\]
    From this, we can conclude that \eqref{eq:cond_very-goodness} holds.
\end{proof}

We remark that the above proof implies that for any \(r\)-color-critical graph \(H\), not necessarily in  \(\mathcal{F}_r\), if its color-reduced multigraph \(H_c\) satisfies the statement of \Cref{thm:ex-gph_stab_r-cc-mgph_k-small}, then \(H\) also satisfies the conclusion of \Cref{lem:ex-gph_stab_cc-gph_with_very-good-cred-mgph}.

\section{Proof of \texorpdfstring{\Cref{thm:ex-gph_4-cc-gph,thm:ex-gph_almost-all_r-cc-gph}}{Theorems 1.3 and 1.4}}
\label{sec:very-good-cred-mgph=>very-good}

In this section, using the results from the previous sections, we prove \Cref{thm:ex-gph_4-cc-gph,thm:ex-gph_almost-all_r-cc-gph}. Let \(H\) be an \(h\)-edge \(r\)-color-critical graph with \(r\geq 4\). By \Cref{prop:almost all in Fr}, we may assume that it is in \(\mathcal{F}_r\) if \(r\geq 5\).
We may further assume that \(H\) has no isolated vertices, so it has at most \(2h\) vertices.
Choose constants \(\varepsilon,\mu >0\) so that we have \(0<1/n \ll \varepsilon \ll \mu \ll 1/h < 1 .\)
Let $k^*= k^*(H) = \frac{r-1}{r-2}(h-1)$.
By \Cref{prop:k-large_enough-to-check_k=k*}, we may assume \(h\leq k\leq \lceil k^*\rceil\).

Assume that \(G\) is simply \(k\)-colored multigraph with 
\[
    e(G) \geq
    \begin{cases*}
        (h-1)\binom{n}{2} & if \(h \leq k < k^*\),\\
        k \cdot t_{r-1}(n) & if \(k\geq k^*\).
    \end{cases*}
\]
By \Cref{prop:reform_to_nested-G} and the discussion after it, assume the coloring of \(G\) is nested. 
In addition, by \Cref{prop:min-deg-condition}, we may assume that the following holds.
\begin{align}\label{eq: mindegcond}
    \delta(G)\geq 
    \begin{cases*}
        (h-1)(n-1) & if  $k < k^*$, \\
        k \cdot d_{r-1}(n) & if $k\geq k^*$.
    \end{cases*}
\end{align}
By \Cref{lem:ex-gph_stab_cc-gph_with_very-good-cred-mgph}, we have one of the following two cases.

\case{\(|G\symmdiff (h-1) K_n| \leq \varepsilon n^2\) and \(k\leq k^*\)}

Note that $k^*\cdot d_{r-1}(n)\geq (h-1)(n-1)$, so  \eqref{eq: mindegcond} implies \(\delta(G)\geq (h-1)(n-1)\) even if \(k=k^*\).
Let 
\[
    E_1\defeq\{e\in P(G):w(e)\ge h\} \quad\text{and}\quad E_2\defeq\{e\in P(G):w(e)\le h-2\}.
\]
As \(|G\symmdiff (h-1)K_n|\le \varepsilon n^2\), we have
	\begin{equation}\label{eq:(thm:very-good-cred-mgph=>very-good_k-small)-pf_few_wrong-edges}
		|E_1\cup E_2| \le \varepsilon n^2.
	\end{equation}
	As \(k\leq k^*<2h-1\), we have
	\(
		0\le e(G)-(h-1)\binom{n}{2}
		\le (k-(h-1))|E_1|-|E_2|
		< h|E_1|-|E_2|,
	\)
	giving 
	\begin{equation} \label{eq:comparison}
	|E_1|>\frac{1}{h}|E_2|.
	\end{equation}
	We claim that either
	\begin{enumerate}[label=(\roman*)]
		\item\label{case:(thm:very-good-cred-mgph=>very-good_k-small)-pf_large-codeg-edge} there exists an edge \(v_1v_2\) and \(A\subseteq V(G)-\{v_1,v_2\}\) of size at least \(\frac{n}{2}-2\) such that \(w(v_1v_2)\ge h\) and \(w(v_iu)\ge h-1\) for all \(i=1,2\) and \(u\in A\), or
		\item\label{case:(thm:very-good-cred-mgph=>very-good_k-small)-pf_large-deg-vertex} there exists a vertex \(v\) and \(B\subseteq V(G)-\{v\}\) of size at least \(\frac{n}{4h}\) such that \(w(vu)\ge h\) for all \(u\in B\).
	\end{enumerate}
	Suppose neither of the two cases hold. We then count the number of the \(2\)-paths each of which consists of an edge in \(E_1\) and an edge in \(E_2\). Because  \ref{case:(thm:very-good-cred-mgph=>very-good_k-small)-pf_large-codeg-edge} does not hold, the number is at least \(|E_1|\cdot \frac{n}{2}\). On the other hand, since \ref{case:(thm:very-good-cred-mgph=>very-good_k-small)-pf_large-deg-vertex} does not hold, the number is at most \(|E_2|\cdot 2\left(\frac{n}{4h}\right)\). However, \eqref{eq:comparison} implies that \(|E_2|\cdot 2\left(\frac{n}{4h}\right)<|E_1|\cdot \frac{n}{2}\), a contradiction.
	
	In either case, by \eqref{eq:(thm:very-good-cred-mgph=>very-good_k-small)-pf_few_wrong-edges} and Tur\'an's theorem, there is a clique \(C\) of size \(|H|< \frac{1}{4h\varepsilon}\) in \(A\) or \(B\) whose edges have multiplicity \(h-1\). Then we can find a multicolored copy of \(H\) in either \(G[\{v_1,v_2\}\cup V(C)]\) or \(G[\{v\}\cup V(C)]\), a contradiction. Therefore \(E_1=E_2=\emptyset\), implying \(G= (h-1)K_n\).

\case{\(|G\symmdiff kT_{r-1}(n)| \leq \varepsilon n^2\) and \(k\geq k^*\)}

By \eqref{eq: mindegcond}, we have \(\delta(G)\geq k d_{r-1}(n)\).
Let $G'$ be a graph on the vertex set \(V(G)\) with \(uv\in E(G')\) if and only if \(w_G(uv)=k\).
Since \(|G\symmdiff kT_{r-1}(n)| \leq \varepsilon n^2\), we know that 
\( e(G')\geq t_{r-1}(n) - \varepsilon n^2> t_{r-2}(n) + \varepsilon n^2. \)
By Erd\H{o}s--Stone--Simonovits theorem, \(G'\) contains a copy \(K\) of an complete \((r-1)\)-partite graph \(K_{2h,2h,\dots, 2h}\) with vertex partition \(W_1,\dots, W_{r-1}\).
It is easy to check that \(K\) is \(H\)-friendly, as any additional vertex \(v\) with \(d_{W_1}(v)\geq 1, d_{W_i}(v)\geq k/2\) for all \(i\geq 2\) yields a multicolored copy of \(H\) within \(K\cup\{v\}\). Moreover, \(K\) is an induced subgraph of \(G\), as additional edge within a color class of \(K\) yields a copy of \(H\).
Hence, \(G\) contains an \(H\)-friendly subgraph with $2h(r-1)$ vertices as an induced subgraph.
Now \Cref{lem:H-friendly-submgph=>(r-1)-partite} implies that \(G\) is \((r-1)\)-partite.
Since \(e(G)\geq kT_{r-1}(n)\), we conclude that \(G=kT_{r-1}(n)\).
\medskip
	
If \(k\neq k^*\), 
we have identified the unique extremal graph as above. 
If \(k=k^*\), we have concluded that \(G\) is either \((h-1)K_n\) or \(k T_{r-1}(n)\). However, the former graph contains less number of edges than \(kt_{r-1}(n)\), thus the latter is the unique \(k\)-extremal graph for \(H\) in this case.

\section{Concluding remarks}

One obvious remaining question is to determine the \(k\)-color extremal numbers for \(r\)-color-critical graphs not in \(\mathcal{F}_r\). Also, determining \(k\)-color extremal numbers for non-color-critical graphs is also a natural question. 
In fact, proof techniques in this paper provides asymptotics of \(k\)-color extremal number of some non-color critical graphs.
Consider an \(r\)-partite \(h\)-edge graph \(H\) with a vertex partition \((X_1,\dots, X_r)\) such that \(\min_{ij\in \binom{r}{2}}e(X_i,X_j)=m\) and \(e(X_i,X_j)\leq \binom{r}{2}^{-1} (h-m) + O( (h-m)/r^4).\)
Additionally assume that for any $(r-1)$-partition $(X'_1,\ldots,X'_r)$ of the vertex set of $H$ the total number of edges in a same part is at least $m$ (i.e., $\sum_{i \in [r]} e(G[X'_i]) \ge m$). Although these conditions might seem artificial at the first sight, some natural graphs like the balanced complete $r$-partite graphs or the so-called generalized book graphs satisfy them.
For such a graph, \(m-1\) copies of \(K_n\) together with \(k-m+1\) copies of identical \(T_{r-1}(n)\) provides a lower bound for \(k\)-color Tur\'an number of \(H\).
The proofs from \Cref{sec:4-vtx-4-cc-mgph_good} and \Cref{sec:r-vtx-mgph_in_Fr_good} can be extended to \(r\)-vertex multigraph with one edge of multiplicity at most \(m\) instead of \(1\). Together with the techinques used in \Cref{sec:mcol-reg-lem}, such a result implies that if \(1/n\ll 1/k\) and \(k\geq \frac{r-1}{r-2}(h-m)+m-1\), then we have 
\(\ex_k(n,H)= (m-1)\binom{n}{2} + (k-m+1)t_{r-1}(n) + o(n^2).\)
For example, this asymptotically determines \(k\)-color extremal number of balanced complete \(r\)-partite graphs.
Note that the assumption \(1/n\ll 1/k\) is necessary here because if \(k\) is large compared to \(n\), as proved in \cite[Theorem~1.1]{keevash2004multicolour}, the \(k\)-color extremal graph consists of \(k\) identical copies of \(\Ex(n,H)\) and has more edges than \((m-1)K_n+ (k-m+1)T_{r-1}(n)\). Determining what relations between \(k\) and \(n\) ensure \(\Ex_k(n,H)= (m-1)K_n+ (k-m+1)T_{r-1}(n)\) is also an interesting question.

\printbibliography


\appendix
\section{Appendix}

\subsection{Proof of \texorpdfstring{\Cref{prop:min-deg-condition}}{Proposition~3.3}}
\label{appdx:(prop:min-deg-condition)-proof}

To prove \Cref{prop:min-deg-condition}, we use the following lemma called the \emph{progressive induction}, a modified form of the induction. See \cite[\S III]{simonovits1968method} for details.
\begin{lemma}[Progressive induction]\label{lem:prog-ind}
	Let \(\calA=\bigsqcup_{n=1}^\infty \calA_n\) be a disjoint union of finite sets. Let \(B\) be a condition or property defined on \(\calA\). Let \(f\colon\NN\cup\calA\to\NN \cup \{0\}\) satisfy the following: 
	\begin{itemize}
		\item If \(a\) satisfies \(B\), then \(f(a)=0\).
		\item There is an \(M_0>0\) such that if \(n>M_0\) and \(a\in\calA_n\), then either \(a\) satisfies \(B\) or there exist an \(n'\) and an \(a'\) such that
		\[
			\frac{n}{2} < n' < n,\quad a'\in\calA_{n'},\quad \text{and}\quad f(a)<f(a').
		\]
	\end{itemize}
	Then there exists an \(n_0>0\) such that if \(n>n_0\), every \(a\in\calA_n\) satisfies \(B\). We can use \(n_0\defeq 2^{s^2}M_0\) where \(s\defeq\max\{f(a):a\in\calA_n, n\le M_0\}+1\).
\end{lemma}

Now \Cref{prop:min-deg-condition} follows from \Cref{prop:min-deg-condition_detailed}, as the multigraphs \(A(n)\) in \Cref{prop:min-deg-condition} satisfies the condition below, i.e., \(e(A(n+1))-\delta(A(n+1))=e(A(n))\).

\begin{proposition}\label{prop:min-deg-condition_detailed}
    Let \(H\) be a (multi)graph and \(k\ge 1\) be fixed. Let \((A(n))_{n=1}^\infty\) be a sequence of simply \(k\)-colored multicolored-\(H\)-free multigraphs such that for each \(n\), we have \(|A(n)|=n\) and \(e(A(n+1))-\delta(A(n+1))=e(A(n))\).
    
    Suppose there is an \(M_0>0\) such that whenever \(n>M_0\) and \(G\) is an extremal simply \(k\)-colored multigraph of order \(n\) with \(\delta(G)\ge \delta(A(n))\), we have \(G=A(n)\). Then there exists an \(n_0=n_0(M_0,k)>0\) such that whenever \(n>n_0\), the multigraph \(A(n)\) is the unique extremal simply \(k\)-colored multigraph of order \(n\).
\end{proposition}
\begin{proof}
    Let \(\calA_n\) be the set of simply \(k\)-colored extremal multigraphs of order \(n\), and let \(B\) be the property such that ``\(G=A(|G|)\)''. Also, define \(f(G)=f(n)\defeq e(G)-e(A(n))\ge 0\) for \(G\in\calA_n\), which is well-defined. Clearly \(f(G)\) vanishes if \(G\) satisfies \(B\). Let \(G\in\calA_n\) with \(n>M_0\), but suppose \(G\) does not satisfy \(B\). Then \(G\neq A(n)\), so \(\delta(G)<\delta(A(n))\). Delete a vertex from \(G\) with degree less than \(\delta(A(n))\) to make \(G'\). Then
    \[
    	e(G')>e(G)-\delta(A(n))=e(G)-(e(A(n))-e(A(n-1))),
    \]
    so \(f(G')>f(G)\). Therefore the conditions of the progressive induction hold, whence \(n_0=2^{s^2}M_0\) with \(s\defeq\max\{f(G):G\in\calA_n, n\le M_0\}+1\) is the desired number. Since \(f(G)\le k\binom{n}{2} < kn^2\) if \(|G|=n\), we can  choose \(n_0=2^{(kM_0^2)^2}M_0\). Hence the statement follows.
\end{proof}

\subsection{Most of the \texorpdfstring{\(r\)}{r}-color-critical graphs are in \texorpdfstring{$\F_r$}{F\textrinferior}}
\label{subsec:most_r-cc-gphs_in_Fr}

In this subsection, we provide a proof sketch of \Cref{prop:almost all in Fr}.

\begin{proof}[Proof sketch of \Cref{prop:almost all in Fr}]
For the convenience of writing, we write $\binom{n}{n_1,\ldots,n_k}$ to denote the multinomial term $\frac{n!}{n_1!\cdots n_k!}$. Observe that this is the number of ways the vertex set $[n]$ can be partitioned into $k$ sets with sizes $n_1,\ldots,n_k$. 
    
Observe that any $r$-color-critical graph can be obtained by adding an edge in one of the parts in an $(r-1)$-partite graph. We call such a graph to have parts $V_1, \ldots, V_{r-1}$ and special part $V_1$, if the corresponding $(r-1)$-partite graph has parts $V_1, \ldots, V_{r-1}$ and the extra edge appears in the first part. Let $F(n_1, \ldots, n_{r-1})$ denote the set of graphs obtained by adding an edge in one of the parts in an \((r-1)\)-partite graphs with part sizes \(n_1,\dots,n_{r-1}\). 
As such a graph can be counted \((r-2)!\) times by permutations over the sets \(V_2,\dots, V_{r-1}\), the number of $r$-color-critical graphs with part sizes $n_1, \ldots, n_{r-1}$ is at most 
\[f(n_1, \ldots, n_{r-1})\defeq |F(n_1, \ldots, n_{r-1})| = \frac{1}{(r-2)!}\binom{n}{n_1,\ldots,n_{r-1}} \cdot \sum_i \binom{n_1}{2} \cdot 2^{\sum_{i\neq j} n_in_j}.\] 
Consequently, the total number of $r$-color-critical graphs is at most $\sum\limits_{\substack{n_1,\ldots,n_{r-1} : \\ n_1+\cdots+n_{r-1}=n}} f(n_1, \ldots, n_{r-1})$. 
As \(2^{\sum n_in_j}\) term decays fast for those non-uniform \(n_1,\dots,n_{r-1}\), one can easily show that the contribution from the terms with non-uniform $(n_1, \ldots, n_{r-1})$ is negligible in this summation. 
In view of these, it is enough to prove that for every $n_1, \ldots, n_{r-1}$ with $n_1 + \cdots + n_{r-1} = n$ and $n_1, \ldots, n_{r-1} = \frac{n}{r-1} \pm n^{5/6}$, almost all graphs in $F(n_1, \ldots, n_{r-1})$: (i) are $r$-color-critical, (ii) can be decomposed into $r-1$ parts in an unique way (up to indexes on \(V_i\) for \(i\geq 2\)), and (iii) belong to the graph family $\F_r$. To prove this, it is sufficient to show the next proposition about the random graph $\GG = \GG(n_1, \ldots, n_{r-1})$, which is defined as follows. The vertex set of $\GG$ is $V_1 \cup \ldots \cup V_{r-1}$ with $|V_i| = n_i$, all the parts have no edges inside them except $V_1$ which has exactly one edge \(e\) chosen uniformly at random among all pairs \(\binom{V_1}{2}\), finally include every edge across different parts independently with probability $1/2$. 

\begin{proposition}
    For every $n_1, \ldots, n_{r-1}$ with $n_1 + \cdots + n_{r-1} = n$ and $n_1, \ldots, n_{r-1} = \frac{n}{r-1} \pm n^{5/6}$, the following facts holds w.h.p.: 
    \begin{enumerate}
    \item the random graph $\GG = \GG(n_1, \ldots, n_{r-1})$ is $r$-color-critical with the unique critical edge $e$, 
    \item there is a unique way to represent $\GG \setminus e$ as an $(r-1)$-partite graph, and 
    \item the random graph $\GG$ belongs to the graph family $\F_r$.
    \end{enumerate}
\end{proposition}

This follows from standard arguments used in random graphs. We sketch a proof giving the essential ideas needed to establish this proposition.

\begin{proofsketch}
    It is clear that the random graph $\GG$ is properly colorable with $r$ colors. Moreover, with a standard first moment argument, one can assert that $\GG$ contains a copy of $K_r$ w.h.p. Thus, w.h.p.\ the chromatic number of $\GG$ is $r$. The extra edge $e$ added in the first part is obviously a critical edge. We claim that this is the only critical edge. With a standard first moment method and union bound, it is easy to show that w.h.p.\ there are two copies of $K_r$ in $\GG$ with the only common edge $e$. This shows that even if we remove any other edge from $\GG$, there would still remain at least a copy of $K_r$, proving our claim.   
    
    To show the second part, it is enough to show that the classical random $(r-1)$-partite graph $G = G_{n_1,\ldots,n_{r-1},1/2}$ can be represented as an $(r-1)$-partite graph in a unique way w.h.p. Let $V_1 \cup \cdots \cup V_{r-1}$ be the vertex set of the random graph $G$, where each edge between $V_i$ and $V_j$ appears independently with probability $1/2$ whenever $i \neq j$. A first moment calculation shows that w.h.p.\ for any $r-1$ vertices $v_1, \ldots, v_{r-1} \not \in V_i$, there exists a common neighbor of those vertices in $V_i$. Call this property $P$. Suppose there is a non-trivial proper $(r-1)$-coloring of the vertex set of $G$ (the trivial coloring is to color each $V_i$ with a distinct color). Consider such a non-trivial coloring. Then, there must exist $i$ such that for all the $r-1$ colors $c$, there is some vertex $v \not \in V_i$ with color $c$. Thus, there is a set of $r-1$ vertices outside of $V_i$ with pairwise distinct colors. By using property $P$, w.h.p.\ there exists a vertex $v \in V_i$ that is adjacent to all those $r-1$ vertices. Thus, the vertex $v$ cannot use any of the $r-1$ colors, a contradiction. Thus, w.h.p.\ the random graph $G$ can be represented as an $(r-1)$-partite graph only in the trivial way. 
    
    With a standard application of Chernoff bound, w.h.p.\ the total number of edges in $\GG$ is at least $\binom{r-1}{2} \frac{n^2}{2(r-1)^2} - O(n^{\frac{23}{12}})$. Another similar application of Chernoff bound together with an union bound shows that w.h.p. the number of edges between any two parts of $\GG$ is at most $\frac{n^2}{2(r-1)^2} + O(n^{\frac{23}{12}})$. Thus, $\GG$ satisfies the defining properties of $\F_r$, proving the third part of our proposition. 
\end{proofsketch}
    
This completes the proof of \Cref{prop:almost all in Fr}.
\end{proof}

\end{document}